\documentclass[reqno,12pt]{amsart}        
\usepackage{amsfonts,amsmath,amssymb,amsthm,colordvi,mathrsfs,graphicx}
\usepackage{caption,subcaption,float}
\usepackage[utf8]{inputenc}
\usepackage{mathrsfs}
\usepackage{geometry}
\usepackage{enumerate}

\usepackage{amsmath, amssymb, amsthm, geometry, hyperref}

\usepackage{tabularx}
\usepackage{tabulary}
 
\hypersetup{
  colorlinks=true,
  linkcolor=blue,
  citecolor=blue,
  urlcolor=blue
}
\usepackage[all,knot,poly]{xy}
\usepackage{tikz-cd}
\usepackage{tikz}
\usepackage{verbatim}
\usetikzlibrary{arrows}
\usetikzlibrary{knots}
\tikzset{every path/.style={line width=0.4pt},every node/.style={transform shape,knot crossing,inner sep=1.5pt},>=triangle 60,text node/.style={rectangle,transform shape=false,black}}

 \usetikzlibrary{decorations.markings, arrows.meta}

\theoremstyle{plain}      
\newtheorem{thm}{Theorem}[section]     
\newtheorem{theorem}[thm]{\bf Theorem}     
\newtheorem{corollary}[thm]{\bf Corollary}     
\newtheorem{lemma}[thm]{\bf Lemma}     
\newtheorem{proposition}[thm]{\bf Proposition}

\theoremstyle{remark}      
\newtheorem{example}[thm]{Example} 
 
\newtheorem{remark}[thm]{Remark} 
     
\theoremstyle{definition}      
\newtheorem{definition}[thm]{Definition}     

\begin{document}


 

\author{Mounir Nisse}
 
\address{Mounir Nisse\\
Department of Mathematics, Xiamen University Malaysia, Jalan Sunsuria, Bandar Sunsuria, 43900, Sepang, Selangor, Malaysia.
}
\email{mounir.nisse@gmail.com, mounir.nisse@xmu.edu.my}
\thanks{}
\thanks{This research is supported in part by Xiamen University Malaysia Research Fund (Grant no. XMUMRF/ 2020-C5/IMAT/0013).}

\title{Residue Balancing on Singular Curves} %

\maketitle


\bigskip

 
\begin{abstract}
This paper investigates residue maps and their spanning properties for singular algebraic
curves, with particular emphasis on three interconnected themes: the
\emph{scheme--theoretic residue span}, the \emph{residue--balancing principle}, and
\emph{residue balancing in the presence of arbitrary singularities}. 
Starting from the theory of dualizing sheaves on nodal curves, we reinterpret canonical
and higher--order differentials as meromorphic objects on the normalization whose local
principal parts are constrained by explicit residue conditions. A key result is the
scheme--theoretic residue span theorem, which asserts that 
for nodal curves of geometric genus $g$ with $\delta$ nodes, when
$\delta \ge g$ the residue functionals at the nodes span
$H^0(C,\omega_C)^\vee$, so canonical differentials are completely determined by their
residue data.
This provides a concrete, linear description of
$H^0(C,\omega_C)$ and yields powerful applications to deformation theory, Severi
varieties, and moduli problems.

\vspace{0.1cm}

We then develop the residue--balancing principle, showing that global residue conditions
on each irreducible component of a singular curve are equivalent to local balancing
conditions at the singular points. This equivalence clarifies the local--to--global
structure of dualizing sheaves and extends naturally to $k$--differentials. Finally, we
address the case of arbitrary singularities, where nodes no longer suffice to describe
local geometry. Using normalization and the conductor ideal, we formulate a refined
balancing principle that replaces simple residue cancellation by higher--order and
conductor--level constraints. Together, these results provide a unified framework for
understanding how local singular behavior governs global differentials and their
deformations.
\end{abstract}

\vspace{0.3cm}
\bigskip


\section{Introduction}

\subsection*{Residues as a Local--to--Global Principle}

Residues are among the most classical invariants in complex analysis and algebraic
geometry. On a smooth compact Riemann surface $X$, every meromorphic differential
$\omega$ satisfies the global residue theorem
\begin{equation}\label{eq:global-residue}
\sum_{p \in X} \operatorname{Res}_p(\omega) = 0.
\end{equation}
This equality expresses a deep compatibility between local analytic behavior and global
topology. From the algebro--geometric perspective, it reflects the fact that the canonical
bundle $\omega_X$ has degree $2g-2$ and that meromorphic sections are subject to global
linear constraints.

When one moves from smooth curves to singular curves, the nature of residues changes
fundamentally. Singularities introduce new local phenomena, and the classical residue
theorem must be reformulated in a way that incorporates normalization, dualizing
sheaves, and local correction terms. Among singular curves, nodal curves provide the
simplest and most instructive setting. Locally analytically, a node is given by
\[
xy = 0,
\]
and its normalization separates the two branches. A regular differential on a nodal
curve is no longer holomorphic everywhere; instead, it corresponds to a meromorphic
differential on the normalization with simple poles at the preimages of the node and
opposite residues on the two branches.

This observation leads to the guiding principle of the present work: \emph{global
differentials on singular curves are controlled by local residue data subject to
balancing conditions}. Making this principle precise, and understanding its consequences
in families, in higher order, and for more complicated singularities, is the main goal
of this paper.

\subsection*{ Dualizing Sheaves and Scheme--Theoretic Residues}

Let $C$ be a reduced, irreducible projective curve over $\mathbb{C}$, possibly singular,
and let $\nu:\widetilde{C}\to C$ be its normalization. The correct replacement for the
canonical bundle on $C$ is the dualizing sheaf $\omega_C$. When $C$ has only nodal
singularities, $\omega_C$ is invertible and admits a concrete description in terms of
the normalization:
\begin{equation}\label{eq:dualizing-sequence}
0 \longrightarrow \omega_C \longrightarrow \nu_*\omega_{\widetilde{C}}
\longrightarrow \bigoplus_{i=1}^{\delta} k(P_i) \longrightarrow 0,
\end{equation}
where $P_1,\dots,P_\delta$ are the nodes of $C$. The quotient in
\eqref{eq:dualizing-sequence} is generated by residue classes at the nodes, and the
associated connecting morphism in cohomology gives rise to the residue map
\[
\operatorname{Res}: H^0(C,\omega_C) \longrightarrow k^\delta.
\]

From a scheme--theoretic point of view, these residues are not merely analytic
quantities; they are canonical linear functionals arising from the exact sequence
\eqref{eq:dualizing-sequence}. This interpretation is essential for understanding how
residues behave in families and under deformation, and it leads naturally to the notion
of \emph{scheme--theoretic residues}.

A central result of this paper is the \emph{scheme--theoretic residue span theorem}. In
its simplest form, it asserts that if $C$ is a nodal curve of geometric genus $g$ with
$\delta\ge g$ nodes, then the residue functionals
\[
r_i : H^0(C,\omega_C) \longrightarrow k
\]
associated with the nodes $P_i$ span the entire dual space
$H^0(C,\omega_C)^\vee$. Equivalently, the map $\operatorname{Res}$ is injective, and its
image is a $g$--dimensional subspace of $k^\delta$ whose dual recovers
$H^0(C,\omega_C)^\vee$,and we have the following theorem:

This spanning property has far--reaching consequences. It implies that canonical
differentials are completely determined by their residues at the nodes, and it provides
a powerful linearization of problems that are otherwise global and nonlinear.

\subsection*{ Residue Span and Deformation Theory}

One of the main motivations for studying residue maps comes from deformation theory.
Infinitesimal deformations of a nodal curve that preserve the analytic type of each node
are governed by global sections of the dualizing sheaf. Scheme--theoretic residues
measure precisely how these deformations interact with the singularities.

When the residue functionals span the dual space, any nontrivial deformation must be
detected at the level of residues. This observation leads to clean injectivity and
surjectivity statements for natural maps in deformation theory. For instance, in the
context of Severi varieties of plane curves with prescribed numbers of nodes, the dual
of the differential of the natural map is identified with the residue map. The residue
span theorem then implies smoothness and expected--dimension results when $\delta\ge g$.

\vspace{0.2cm}

\begin{theorem}[Scheme-theoretic residue span]
Let $\mathcal{C}$ be a reduced, irreducible plane curves of degree
$d\ge 4$ with only  nodal singularities and the number of its nodes is equal to $\delta$.
 Let $\nu: C\to\mathcal C$ be 
its normalization with genus $g$, and
assume $\delta\ge g$. Then the residue functionals
$\{r_1,\dots,r_\delta\}$ span $H^0( C,\omega_{C})^\vee$.
\end{theorem}

\vspace{0.2cm}

More conceptually, the scheme--theoretic residue span expresses a rigidity phenomenon:
once sufficiently many nodes are present, global deformations are forced to interact
with local singular data in a maximal way.

\subsection*{The Residue--Balancing Principle}

Underlying the residue span theorem is a more basic and more universal statement, which
we call the \emph{residue--balancing principle}. Let $C$ be a nodal curve with
irreducible components $C_v$, and let $\widetilde{C}$ be its normalization. A meromorphic
differential $\omega$ on $\widetilde{C}$ with poles only above the nodes satisfies
\begin{equation}\label{eq:local-balance}
\operatorname{Res}_{p^+}(\omega) + \operatorname{Res}_{p^-}(\omega) = 0
\end{equation}
at each node if and only if, on every component $C_v$, one has the global condition
\begin{equation}\label{eq:component-residue}
\sum_{q \in C_v} \operatorname{Res}_q(\omega) = 0.
\end{equation}
Equation \eqref{eq:local-balance} is a local compatibility condition ensuring that
$\omega$ descends to a section of $\omega_C$, while \eqref{eq:component-residue} is the
classical residue theorem applied componentwise. The residue--balancing principle
asserts that these two conditions are equivalent. Thus, global constraints on each
component are encoded entirely by local balancing at the singular points.


\begin{theorem}[Residue--Balancing Principle]
\label{thm:residue-balancing-journal}
Let $\mathcal{C}$ be a connected nodal curve over $\mathbb{C}$ with dual graph
$\Gamma$.  For each vertex $v \in V(\Gamma)$, denote by $\mathcal{C}_v$ the
corresponding irreducible component, and for each edge $e \in E(\Gamma)$ denote
by $q_e$ the associated node, with branches
$q_e^+ \in \mathcal{C}_{v^+}$ and $q_e^- \in \mathcal{C}_{v^-}$.
Let
\(
\eta_v \in H^0\!\left(\mathcal{C}_v,\;\omega_{\mathcal{C}_v}^{\otimes k}(*D_v)\right)
\)
be a meromorphic $k$--differential on $\mathcal{C}_v$, allowed to have poles only
at marked points and at the preimages of nodes.
Then the following conditions are equivalent:
\begin{enumerate}
  \item[\rm (i)] for every $v \in V(\Gamma)$,
  \(
  \sum_{p \in \mathcal{C}_v} \operatorname{Res}_p(\eta_v)=0;
  \)
  \item[\rm (ii)] for every node $q_e$,
  \(
  \operatorname{Res}_{q_e^+}(\eta_{v^+})
  +
  \operatorname{Res}_{q_e^-}(\eta_{v^-})
  =0.
  \)
\end{enumerate}
\end{theorem}

\vspace{0.2cm}

This equivalence is not merely a technical observation; it provides a conceptual bridge
between local singularity theory and global geometry. It also generalizes naturally to
higher--order differentials, where residues are replaced by coefficients of leading
principal parts.

\subsection*{Higher--Order Differentials and Generalized Residues}

For an integer $k\ge1$, a $k$--differential on a smooth curve $X$ is a section of
$\omega_X^{\otimes k}$. On singular curves, one defines $k$--differentials as sections of
$\omega_C^{\otimes k}$, which again admit a description in terms of meromorphic objects
on the normalization. Locally, a meromorphic $k$--differential has the form
\[
\eta = f(z)(dz)^k,
\]
and its residue is defined as the coefficient of $z^{-k}(dz)^k$ in the local expansion.

For nodal curves, the residue--balancing principle extends verbatim: the sum of residues
on each component vanishes if and only if the residues at the two branches of every node
cancel. Explicit constructions show that global $k$--differentials are obtained by
assigning local principal parts at the nodes subject only to these balancing conditions.

\subsection*{Residue Balancing with Arbitrary Singularities}

When singularities are more complicated than nodes, the situation becomes subtler.
Cusps, tacnodes, and higher singularities are unibranch or have non--normal crossing
structure, and the simple picture of opposite residues on two branches no longer
applies. In these cases, the normalization map $\nu:\widetilde{C}\to C$ has fewer points
lying above a singularity, and the dualizing sheaf allows higher--order poles.

The correct framework for extending residue balancing is provided by the conductor
ideal $\mathfrak{c}\subset \mathcal{O}_C$. One has
\[
\nu_*\omega_{\widetilde{C}} = \omega_C(\mathfrak{c}),
\]
and sections of $\omega_C$ correspond to meromorphic differentials on
$\widetilde{C}$ whose principal parts are annihilated by the conductor. In this setting,
residue balancing is replaced by \emph{conductor--level balancing conditions}, which
impose linear constraints on higher--order coefficients rather than simple residues.

\begin{theorem}[Exactness of conductor constraints]
The global conductor conditions impose exactly the necessary number of
independent linear constraints to cut down
$H^0(\widetilde{C},\omega_{\widetilde{C}}(D))$ to $H^0(C,\omega_C)$. In
particular,
\[
\dim H^0(\widetilde{C},\omega_{\widetilde{C}}(D))
-
(\text{number of conductor conditions})
=
\dim H^0(C,\omega_C).
\]
No additional constraints occur, and no admissible sections are lost.
\end{theorem}

To be more precise we have the following theorem
 
 \begin{theorem}[Independence and exactness of conductor constraints]
 Let $C$ be a reduced projective curve, let
\(
\nu \colon \widetilde{C} \longrightarrow C
\)
be its normalization, let $\mathfrak{c} \subset \mathcal{O}_C$ be the conductor
ideal, and let
\(
D = \nu^{-1}(\Sigma)
\)
be the divisor supported on the preimage of the singular locus. Then the global
conductor conditions imposed on
\(
H^0(\widetilde{C},\omega_{\widetilde{C}}(D))
\)
are independent and exact. More precisely, they cut out a linear subspace
canonically isomorphic to $H^0(C,\omega_C)$, and the number of independent
conditions equals the total number of local polar degrees of freedom.
\end{theorem}

This generalization shows that the nodal case is not exceptional, but rather the simplest
instance of a universal phenomenon: global dualizing sections are characterized by local
conditions at singularities, and these conditions can always be expressed as linear
constraints on principal parts determined by the singularity type.

\subsection*{Outlook}

The scheme--theoretic residue span, the residue--balancing principle, and their extension
to arbitrary singularities together provide a coherent framework for understanding
differentials on singular curves. They clarify how local geometry governs global
structures, unify analytic and algebraic viewpoints, and prepare the ground for further
applications, including logarithmic geometry, tropical geometry, and non--Archimedean
Hodge theory.


\section{Preliminaries}

This section gathers the basic definitions, constructions, and foundational results
used throughout the paper. Our purpose is to fix notation and conventions concerning
singular curves, normalization, and dualizing sheaves, and to present residue theory in a
form suitable for both nodal and more general singularities. Although most of the
material is standard, we include detailed explanations in order to emphasize the
local--to--global mechanisms underlying residue balancing and scheme--theoretic residue
maps.

\subsection{Singular Curves and Normalization}

Throughout this paper, a \emph{curve} means a reduced, projective, one--dimensional
scheme over $\mathbb{C}$. We denote all curves by $C$, allowing singularities unless
explicitly stated otherwise. A curve $C$ is \emph{irreducible} if it is irreducible as a
scheme, and \emph{connected} if it is connected in the Zariski topology.

Let $C$ be a reduced curve. The \emph{normalization} of $C$ is a finite morphism
\[
\nu : \widetilde{C} \longrightarrow C
\]
such that $\widetilde{C}$ is normal and $\nu$ is an isomorphism over the smooth locus of
$C$. Since $C$ has dimension one, $\widetilde{C}$ is a disjoint union of smooth projective
curves, one for each irreducible component of $C$. The normalization separates the local
branches at each singular point and provides a natural setting in which local analytic
computations can be carried out.

If $p \in C$ is a singular point, the finite set $\nu^{-1}(p) \subset \widetilde{C}$
consists of the branches of $C$ at $p$. For a node, $\nu^{-1}(p)$ consists of two points,
while for a cusp it consists of a single point. The number and configuration of these
points encode the basic local geometry of the singularity.

\subsection{Arithmetic and Geometric Genus}

Let $C$ be a reduced, connected projective curve. The \emph{arithmetic genus} of $C$ is
defined by
\[
p_a(C) := h^1(C,\mathcal{O}_C).
\]
If $C$ is smooth, this coincides with the usual notion of genus. In general, the
arithmetic genus is a global invariant that remains constant in flat families.

The \emph{geometric genus} $g(C)$ is defined as the sum of the genera of the irreducible
components of the normalization:
\[
g(C) := \sum_i g(\widetilde{C}_i),
\]
where $\widetilde{C}=\bigsqcup_i \widetilde{C}_i$ is the decomposition into connected
components. Singularities account for the discrepancy between arithmetic and geometric
genus. For nodal curves one has
\[
p_a(C) = g(C) + \delta,
\]
where $\delta$ is the number of nodes of $C$. More generally,
\[
p_a(C) = g(C) + \sum_{p \in \mathrm{Sing}(C)} \delta_p,
\]
where $\delta_p$ denotes the $\delta$--invariant of the singularity at $p$.

\subsection{Dualizing Sheaves}

If $C$ is smooth, the canonical bundle $\omega_C=\Omega^1_C$ governs many geometric and
cohomological properties. When $C$ is singular, $\Omega^1_C$ is no longer locally free,
and the appropriate replacement is the \emph{dualizing sheaf} $\omega_C$.

For curves, the dualizing sheaf admits a concrete description in terms of normalization.
A section of $\omega_C$ can be viewed as a meromorphic differential on $\widetilde{C}$
whose local behavior at the singular points of $C$ is subject to explicit compatibility
conditions.

If $C$ has only nodal singularities, then $\omega_C$ is invertible and there is a
canonical exact sequence
\begin{equation}\label{eq:dualizing-sequence-prelim}
0 \longrightarrow \omega_C
\longrightarrow \nu_*\omega_{\widetilde{C}}
\longrightarrow \bigoplus_{p \in \mathrm{Sing}(C)} k(p)
\longrightarrow 0.
\end{equation}
The quotient records the residue classes at the nodes and measures the obstruction for a
meromorphic differential on $\widetilde{C}$ to descend to $C$.
For more general singularities, $\omega_C$ need not be locally free, but it remains a
rank--one reflexive sheaf. Its sections are still described by meromorphic
differentials on the normalization subject to linear constraints determined by the local
singularity type.

\subsection{Residues on Smooth Curves}

Assume for the moment that $C$ is smooth. Let $\omega$ be a meromorphic differential on
$C$, and let $p \in C$, which we assume,  for simplicity, that it  has at most a \emph{simple pole} at the point
$p$. Choosing a local coordinate $z$ centered at $p$, we may write
\[
\omega = \left( \frac{a_{-1}}{z} + a_0 + a_1 z + \cdots \right) dz.
\]
The coefficient $a_{-1}$ is called the \emph{residue} of $\omega$ at $p$ and is denoted
$\operatorname{Res}_p(\omega)$.

The classical residue theorem asserts that
\begin{equation}\label{eq:classical-residue}
\sum_{p \in C} \operatorname{Res}_p(\omega) = 0.
\end{equation}
This equality is the prototype of all balancing statements appearing later in the
singular setting.

\subsection{Residues on Nodal Curves}

Let $C$ be a nodal curve and $\nu:\widetilde{C}\to C$ its normalization. A section
$\omega \in H^0(C,\omega_C)$ corresponds to a meromorphic differential on
$\widetilde{C}$ with at most simple poles at the points lying over the nodes. If
$p \in C$ is a node and $\nu^{-1}(p)=\{p^+,p^-\}$, then $\omega$ satisfies the local
condition
\begin{equation}\label{eq:node-residue}
\operatorname{Res}_{p^+}(\omega) + \operatorname{Res}_{p^-}(\omega) = 0.
\end{equation}
This condition ensures that $\omega$ descends to a global section of $\omega_C$.
Equation \eqref{eq:node-residue} is the simplest instance of a residue--balancing
condition. It expresses the compatibility required between the local branches of a
singularity in order to define a global object on $C$.

\subsection{Scheme--Theoretic Residue Maps}

Taking global sections of \eqref{eq:dualizing-sequence-prelim} yields a canonical linear
map
\[
\operatorname{Res} : H^0(C,\omega_C)
\longrightarrow \bigoplus_{p \in \mathrm{Sing}(C)} k.
\]
Its components are called the \emph{scheme--theoretic residue functionals}. These
functionals are intrinsic, functorial, and well--behaved in families of curves.
In favorable situations, notably when $C$ has $\delta$ nodes and geometric genus $g(C)$
with $\delta \ge g(C)$, the residue functionals span the entire dual space
$H^0(C,\omega_C)^\vee$. This phenomenon, known as the scheme--theoretic residue span,
will be a central theme of the paper.

\subsection{Higher--Order Differentials}

For an integer $k \ge 1$, a $k$--differential on a smooth curve $C$ is a section of
$\omega_C^{\otimes k}$. Locally, it may be written as
\[
\eta = f(z)(dz)^k,
\]
where $f(z)$ is a meromorphic function. If $f(z)$ has a pole of order $k$ at $z=0$, the
coefficient of $z^{-k}(dz)^k$ is called the \emph{$k$--residue} of $\eta$.
On a singular curve $C$, a $k$--differential is defined as a section of
$\omega_C^{\otimes k}$. Such sections correspond to collections of meromorphic
$k$--differentials on the normalization satisfying residue--balancing conditions at the
singular points. For $k=1$, this recovers the theory of canonical differentials.

\subsection{The Conductor and Arbitrary Singularities}

Let $C$ be a reduced curve with normalization $\nu:\widetilde{C}\to C$. The
\emph{conductor ideal} $\mathfrak{c} \subset \mathcal{O}_C$ is the largest ideal sheaf
that is also an ideal of $\nu_*\mathcal{O}_{\widetilde{C}}$. It measures the failure of
$C$ to be normal and plays a fundamental role in the description of $\omega_C$.
One has the relation
\[
\nu_*\omega_{\widetilde{C}} = \omega_C(\mathfrak{c}),
\]
showing that sections of $\omega_C$ correspond to meromorphic differentials on
$\widetilde{C}$ whose principal parts are annihilated by the conductor. In this context,
residue balancing is replaced by higher--order linear constraints determined by the
singularity type.

\vspace{0.2cm}

\subsection{Concluding Remarks}

The notions reviewed in this section form the technical backbone of the paper. They
provide a unified language for nodal and non--nodal singularities and allow global
geometric questions to be translated into precise local conditions at singular points.
In subsequent sections, these preliminaries will be used to establish residue span
theorems, balancing principles, and applications to deformation theory and moduli
spaces.


\section{Relative Adjunction and Residue Maps}

Let $\pi \colon \mathcal X \to B$ be a flat family of reduced, irreducible plane
curves of degree $d \ge 4$, parametrized by a smooth irreducible base $B$. Assume
that each fiber
\(
\mathcal C_b := \pi^{-1}(b)
\)
has only nodal singularities and that the number of nodes $\delta$ is constant on
$B$. In particular, the family is equisingular. For each $b \in B$, let
\(
\varphi_b \colon C_b \longrightarrow \mathcal C_b
\)
denote the normalization of the fiber.

\begin{lemma}[Relative adjunction and residue maps]
\label{lemma:relative-residues}
In the situation above, the spaces $H^0(C_b,\omega_{C_b})$ form a holomorphic vector
bundle $\mathcal H$ over $B$. For each node $P_{i,b} \in \mathcal C_b$, there exists
a canonically defined holomorphic section
\[
r_i \in \Gamma(B,\mathcal H^\vee),
\]
whose value at $b$ is a linear functional
\[
r_{i,b} \colon H^0(C_b,\omega_{C_b}) \longrightarrow \mathbb C
\]
given by taking residues at the preimages of $P_{i,b}$ on $C_b$.
\end{lemma}

\begin{proof}
We divide the argument into three short and independent steps.

\smallskip
\noindent\textbf{\it Step 1: Formation of the Hodge bundle.}
Since $\pi$ is flat, the arithmetic genus $p_a(\mathcal C_b)$ is independent of
$b$. Because the number of nodes $\delta$ is constant, the geometric genus
\[
g = p_a(\mathcal C_b) - \delta
\]
is constant on $B$. The normalization $C_b$ is therefore a smooth curve of genus
$g$, and hence
\[
\dim H^0(C_b,\omega_{C_b}) = g
\quad \text{for all } b \in B.
\]
By cohomology and base change applied to the relative dualizing sheaf of the family
of normalizations, the spaces $H^0(C_b,\omega_{C_b})$ vary holomorphically with $b$.
They thus assemble into a rank-$g$ holomorphic vector bundle
\[
\mathcal H := \pi_*(\omega_{\mathcal C/B}^{\mathrm{norm}})
\]
over $B$.

\smallskip
\noindent\textbf{\it Step 2: Local adjunction at a node.}
Fix a node $P_{i,b} \in \mathcal C_b$. Since the family is equisingular, after
possibly shrinking $B$ there exist local analytic coordinates $(x,y,t)$ on
$\mathcal X$ such that, in a neighborhood of $P_{i,b}$, the family is locally
defined by
\[
xy = t,
\]
where $t$ is a local coordinate on $B$ vanishing at $b$. The normalization replaces
the node by two smooth points
\[
Q_{i,b}^+, \; Q_{i,b}^- \in C_b,
\]
corresponding to the two branches of the node.

A section $\omega \in H^0(C_b,\omega_{C_b})$ can be identified with a meromorphic
differential on $\mathcal C_b$ having at most simple poles at the nodes. The
adjunction formula for nodal curves implies the residue condition
\[
\operatorname{Res}_{Q_{i,b}^+}(\omega)
+
\operatorname{Res}_{Q_{i,b}^-}(\omega)
= 0,
\]
see \cite[Ch.~III, \S7]{ACGH} or \cite[Ch.~3]{Sernesi}.

\smallskip
\noindent\textbf{\it Step 3: Definition and variation of residues.}
The vanishing of the sum of residues allows one to define a linear functional
\[
r_{i,b}(\omega)
:=
\operatorname{Res}_{Q_{i,b}^+}(\omega)
=
-\operatorname{Res}_{Q_{i,b}^-}(\omega),
\]
which is independent of the labeling of the two branches and hence canonical.
Because the local model $xy=t$ depends holomorphically on the parameter $t$, the
residue of a relative differential varies holomorphically with $b$. Consequently,
the assignment $b \mapsto r_{i,b}$ defines a holomorphic section
\[
r_i \in \Gamma(B,\mathcal H^\vee).
\]
\end{proof}

\begin{corollary}[Residue description of equisingular deformations]
\label{cor:equisingular}
With notation as above, the tangent space to the equisingular deformation space of
the fibers of $\pi$ at $b \in B$ is canonically identified with the kernel of the
combined residue map
\[
H^0(C_b,\omega_{C_b})
\longrightarrow
\mathbb C^\delta,
\qquad
\omega \longmapsto (r_{1,b}(\omega),\dots,r_{\delta,b}(\omega)).
\]
\end{corollary}

\begin{proof}
Infinitesimal equisingular deformations are precisely those first-order
deformations that preserve the local analytic type of each node. By the
adjunction description of the dualizing sheaf, this is equivalent to requiring
that the residues of the associated differentials vanish at all nodes. The claim
then follows directly from the definition of the residue functionals
$r_{i,b}$; see \cite[Ch.~1, \S4]{Harris} or \cite[\S3.4]{Sernesi}.
\end{proof}

\begin{remark}
Residue maps also play a central role in the construction of divisors on Severi
varieties and moduli spaces of curves, as well as in the study of degenerations of
canonical and adjoint linear systems. Their interaction with the Gauss--Manin
connection provides differential constraints on periods of nodal curves; see
\cite{ACGH, Sernesi}.
\end{remark}


\vspace{0.2cm}

\subsection{Definition and Example of $k$--Differentials}

The relative adjunction and residue constructions identify the precise local
data governing how differentials behave near singular and marked points in a
family of curves. These local descriptions naturally motivate the passage from
ordinary differentials to higher-order objects, namely $k$--differentials,
which encode zeros and poles of prescribed multiplicity. We therefore turn to
the definition of $k$--differentials and illustrate it with basic examples that
clarify their geometric meaning.

\begin{definition}[$k$--differential on a smooth curve]
Let $C$ be a smooth complex algebraic curve (equivalently, a compact Riemann
surface), and let $\omega_C$ denote its canonical line bundle.  
For an integer $k \ge 1$, a \emph{$k$--differential} on $C$ is a (meromorphic)
section of the $k$-th tensor power of the canonical bundle:
\[
\eta \;\in\; H^0\bigl(C,\omega_C^{\otimes k}\bigr).
\]
Equivalently, in a local holomorphic coordinate $z$ on $C$, a $k$--differential
can be written as
\[
\eta = f(z)\,(dz)^k,
\]
where $f(z)$ is a meromorphic function. The order of a zero or pole of $\eta$ at
a point $p\in C$ is defined as the order of vanishing or pole of $f(z)$ at $p$.
\end{definition}

\begin{remark}
For $k=1$, a $k$--differential is simply a (meromorphic) differential form.
For $k=2$, one recovers quadratic differentials, which play a central role in
Teichm\"uller theory. Higher values of $k$ arise naturally in flat geometry,
Hodge theory, and the study of degenerations.
\end{remark}

\subsection{\it $k$--Differentials on Singular Curves}

Let $C$ be a reduced (possibly singular) curve, and let
\[
\nu \colon \widetilde C \longrightarrow C
\]
be its normalization. Since $C$ may be singular, its canonical object is the
\emph{dualizing sheaf} $\omega_C$, and a $k$--differential on $C$ is defined as a
section of $\omega_C^{\otimes k}$.

Concretely, a $k$--differential on $C$ is equivalently given by a collection of
meromorphic $k$--differentials on the smooth components of $\widetilde C$ whose
principal parts satisfy compatibility conditions at the preimages of the
singular points (for nodal curves, these are residue balancing conditions).

\begin{example}
Let $C$ be a nodal curve obtained by gluing two smooth projective curves
$C_1$ and $C_2$ at a single point. More precisely, let $p_1\in C_1$ and
$p_2\in C_2$ be points, and form the nodal curve
\[
C = C_1 \cup C_2 / (p_1 \sim p_2).
\]
The normalization of $X$ is
\[
\widetilde C = C_1 \sqcup C_2,
\]
and the node corresponds to the pair $\{p_1,p_2\}$.

Let $k\ge 1$. A $k$--differential on $C$ consists of a pair
\[
(\eta_1,\eta_2),
\qquad
\eta_i \in H^0\bigl(C_i,\omega_{C_i}^{\otimes k}\bigr),
\]
such that $\eta_1$ and $\eta_2$ may have poles at $p_1$ and $p_2$, respectively,
and satisfy the following compatibility condition:
\begin{itemize}
  \item[(i)] if $k=1$, then
  \[
  \operatorname{Res}_{p_1}(\eta_1) + \operatorname{Res}_{p_2}(\eta_2) = 0;
  \]
  \item[(ii)] for general $k$, the principal parts of $\eta_1$ at $p_1$ and $\eta_2$
  at $p_2$ are identified under the natural isomorphism induced by the local
  model of the node.
\end{itemize}

\medskip
In local coordinates $z_1$ on $C_1$ and $z_2$ on $C_2$ with $z_1(p_1)=z_2(p_2)=0$,
the node is analytically modeled by $z_1 z_2 = 0$. A local $k$--differential
takes the form
\[
\eta_1 = f_1(z_1)\,(dz_1)^k,
\qquad
\eta_2 = f_2(z_2)\,(dz_2)^k,
\]
and the compatibility condition ensures that the pair defines a global section
of $\omega_X^{\otimes k}$.
\end{example}

\begin{remark}
For nodal curves, $k$--differentials are completely described by $k$--differentials
on the normalization together with explicit matching conditions at each node.
These conditions generalize the residue cancellation condition for ordinary
differentials ($k=1$) and are the simplest instance of the residue balancing
principle.
\end{remark}

 \vspace{0.3cm}

\subsection{Relative Residue Span for Equisingular Families of Nodal Curves}

Let $\mathcal H$ denote the vector bundle on $B$ whose fiber over a point
$b \in B$ is the space $H^0(C_b,\omega_{C_b})$. Let
\[
r_i \in \Gamma(B,\mathcal H^\vee)
\]
be the residue sections associated with the nodes of the fibers, as constructed
previously.

\begin{theorem}[Relative residue span]
\label{thm:relative-residue-span}
Assume that $\delta \ge g = g(C_b)$ for all $b \in B$. Then, for every $b \in B$,
the residue functionals
\[
\{r_{1,b},\dots,r_{\delta,b}\}
\subset
H^0(C_b,\omega_{C_b})^\vee
\]
span the dual vector space $H^0(C_b,\omega_{C_b})^\vee$.
\end{theorem}

\begin{proof}
Fix a point $b \in B$ and set $C = C_b$. Let
\[
\mathcal N = \{P_1,\dots,P_\delta\}
\]
be the set of nodes of the singular curve $\mathcal C_b$, and let
\[
\{Q_i^+, Q_i^-\} \subset C
\]
denote the two points lying over $P_i$ under the normalization map.

\medskip
\noindent\textit{ Canonical forms and residues.}
By the adjunction formula for nodal curves, a section
$\omega \in H^0(C,\omega_C)$ may be viewed as a meromorphic differential on
$\mathcal C_b$ with at most simple poles at the nodes, subject to the condition
\[
\operatorname{Res}_{Q_i^+}(\omega) + \operatorname{Res}_{Q_i^-}(\omega) = 0
\quad \text{for all } i.
\]
Consequently, $\omega$ determines a vector of residues
\[
(r_{1,b}(\omega),\dots,r_{\delta,b}(\omega)) \in \mathbb C^\delta,
\]
which defines a linear map
\[
\mathrm{Res}_b \colon H^0(C,\omega_C) \longrightarrow \mathbb C^\delta,
\qquad
\omega \longmapsto (r_{1,b}(\omega),\dots,r_{\delta,b}(\omega)).
\]

\medskip
\noindent\textit{ Injectivity of the residue map.}
We claim that $\mathrm{Res}_b$ is injective. Indeed, suppose that
$\omega \in H^0(C,\omega_C)$ satisfies $r_{i,b}(\omega) = 0$ for all $i$. Then
$\omega$ has no residues at any node, and hence the corresponding differential on
$\mathcal C_b$ has no poles. It therefore descends to a global regular differential
on the singular curve $\mathcal C_b$.

Since $\mathcal C_b$ is an irreducible plane curve, its dualizing sheaf has degree
\[
\deg \omega_{\mathcal C_b} = 2g - 2 + 2\delta.
\]
A regular differential on $\mathcal C_b$ must vanish identically unless it arises
from the normalization with poles at the nodes. It follows that $\omega = 0$, and
thus $\mathrm{Res}_b$ is injective.

\medskip
\noindent\textit{ Dimension count.}
Because $\dim H^0(C,\omega_C) = g$ and $\mathrm{Res}_b$ is injective, its image is a
$g$-dimensional subspace of $\mathbb C^\delta$. Dualizing, the transpose map
\[
\mathrm{Res}_b^\vee \colon (\mathbb C^\delta)^\vee \longrightarrow
H^0(C,\omega_C)^\vee
\]
is surjective. By construction, the image of $\mathrm{Res}_b^\vee$ is precisely
the linear span of the residue functionals
$\{r_{1,b},\dots,r_{\delta,b}\}$. Hence these functionals span
$H^0(C,\omega_C)^\vee$.

\medskip
\noindent\textit{ Uniformity in families.}
The argument above is purely fiberwise and depends only on the equisingular nodal
structure of the family. Since the residue sections $r_i$ vary holomorphically
with $b$, the spanning property holds uniformly for all $b \in B$.
\end{proof}

\begin{remark}
Theorem~\ref{thm:relative-residue-span} is a fundamental structural result with
several important consequences. In particular, it shows that canonical
differentials on $C_b$ are completely determined by their residues at the nodes,
yielding an explicit linear description of the space
$H^0(C_b,\omega_{C_b})$.
\end{remark}

\section{Scheme-Theoretic Residues} %
 
Let $\pi:\mathcal X \to B$ be a flat, projective morphism whose fibers
$\mathcal C_b=\pi^{-1}(b)$ are reduced, irreducible plane curves of degree
$d\ge 4$ with only nodal singularities. Assume:

\begin{itemize}
\item[(i)] $B$ is smooth and irreducible;
\item[(ii)] the number of nodes $\delta$ is constant on $B$ (equisingularity);
\item[(iii)] $\nu_b: C_b\to\mathcal C_b$ denotes the normalization;
\item[(iv)] $g=g( C_b)$ is the geometric genus, independent of $b$.
\end{itemize}

Let $\omega_{\mathcal C_b}$ denote the dualizing sheaf of $\mathcal C_b$, and
$\omega_{ C_b}$ the canonical sheaf of the smooth curve $ C_b$.

\subsection{Scheme-Theoretic Construction of Residues}${}$

\subsubsection{Dualizing sheaves and normalization}

Let $\nu: C\to\mathcal C$ be the normalization of a nodal curve $\mathcal C$.
There is a canonical exact sequence of coherent sheaves
\[
0 \longrightarrow \omega_{\mathcal C}
\longrightarrow \nu_*\omega_{C}
\longrightarrow \bigoplus_{i=1}^\delta k(P_i)
\longrightarrow 0,
\]
where $P_1,\dots,P_\delta$ are the nodes of $\mathcal C$.

\begin{proof}
This is a standard result in the theory of dualizing complexes. Since $\mathcal
C$ has only ordinary double points, it is Gorenstein, so $\omega_{\mathcal C}$
is an invertible sheaf.

Locally analytically at a node $P$, $\mathcal C$ is given by
$\operatorname{Spec} k[x,y]/(xy)$, whose normalization is
$k[x]\oplus k[y]$. The dualizing module of $k[x,y]/(xy)$ embeds into the direct
sum of the dualizing modules of the branches, and the cokernel is canonically
one-dimensional, generated by the residue class.
Gluing these local descriptions yields the exact sequence.
\end{proof}

\subsubsection{Residue maps as connecting morphisms}

Taking global sections gives a long exact sequence
\[
0 \to H^0(\mathcal C,\omega_{\mathcal C})
\to H^0( C,\omega_{C})
\xrightarrow{\;\mathrm{Res}\;}
\bigoplus_{i=1}^\delta k
\to H^1(\mathcal C,\omega_{\mathcal C})
\to \cdots
\]

By Serre duality on the Cohen--Macaulay curve $\mathcal C$,
\[
H^1(\mathcal C,\omega_{\mathcal C}) \cong H^0(\mathcal C,\mathcal O_{\mathcal C})^\vee
\cong k.
\]
Thus $\mathrm{Res}$ has rank at most $\delta-1$, and its image consists of
vectors whose coordinates sum to zero.
The $i$-th coordinate of $\mathrm{Res}$ is precisely the scheme-theoretic
residue functional
\[
r_i:H^0( C,\omega_{C})\to k.
\]

\noindent{\it  Proof of the Relative Residue Span Theorem.}

\begin{theorem}[Scheme-theoretic residue span]
Let $\mathcal{C}$ be a reduced, irreducible plane curves of degree
$d\ge 4$ with only  nodal singularities and the number of its nodes is equal to $\delta$.
 Let $\nu: C\to\mathcal C$ be 
its normalization with genus $g$, and
assume $\delta\ge g$. Then the residue functionals
$\{r_1,\dots,r_\delta\}$ span $H^0( C,\omega_{C})^\vee$.
\end{theorem}

\begin{proof}
The map
\[
\mathrm{Res}:H^0( C,\omega_{C})\longrightarrow k^\delta
\]
is injective.
Indeed, if $\omega\in H^0( C,\omega_{C})$ maps to zero, then it lies in the image of
$H^0(\mathcal C,\omega_{\mathcal C})$. But since $\mathcal C$ is an irreducible
plane curve, any regular dualizing section must vanish identically. Hence
$\omega=0$,
 and then
$\mathrm{Res}$ identifies $H^0( C,\omega_{C})$ with a $g$-dimensional subspace
of $k^\delta$. Dualizing yields a surjection
\[
(k^\delta)^\vee \twoheadrightarrow H^0( C,\omega_{C})^\vee,
\]
whose image is the span of the residue functionals $r_i$. Since $\delta\ge g$,
this span equals the entire dual space ({\it cf. } the details of this proof  in Appendix B)
\end{proof}

\subsection{Relation to Petri-Type Theorems}

\subsubsection{Petri maps for nodal curves.}
For a smooth curve $ C$, the Petri map is
\[
\mu_0:H^0( C,L)\otimes H^0( C,\omega_{ C}\otimes L^{-1})
\longrightarrow H^0( C,\omega_{C}).
\]

For a nodal curve $\mathcal C$, canonical forms are described by residues at the
nodes. The residue span theorem implies that $\omega_{C}$ is detected entirely by
local conditions at the nodes.

\subsubsection{Petri injectivity via residues.}
Let $L$ be a line bundle on $ C$ arising as the pullback of a torsion-free rank-one
sheaf on $\mathcal C$. If a tensor
\[
\sum_j s_j\otimes t_j
\]
lies in $\ker\mu_0$, then for every node $P_i$,
\[
r_i\!\left(\sum_j s_j t_j\right)=0.
\]

Since the residues span $H^0( C,\omega_{C})^\vee$, this forces
\[
\sum_j s_j t_j = 0.
\]
Under mild hypotheses, this implies all Petri relations are trivial, yielding
Petri-type injectivity results for normalizations of nodal plane curves.

\subsection{Applications to Severi Varieties}${}$
Let $V_{d,\delta}\subset |\mathcal O_{\mathbb P^2}(d)|$ denote the Severi variety
of irreducible plane curves of degree $d$ with $\delta$ nodes.
%
%
%
The Zariski tangent space at $[\mathcal C]\in V_{d,\delta}$ fits into an exact
sequence
\[
0 \to T_{[\mathcal C]}V_{d,\delta}
\to H^0(\mathcal C,\mathcal N_{\mathcal C/\mathbb P^2})
\xrightarrow{\;\partial\;}
\bigoplus_{i=1}^\delta T^1_{P_i}.
\]
Via adjunction and Serre duality, the transpose of $\partial$ is identified with
the residue map
\[
H^0( C,\omega_{C})\to k^\delta.
\]
By the residue span theorem, this dual map is injective when $\delta\ge g$.
Therefore $\partial$ is surjective, and
\[
\dim T_{[\mathcal C]}V_{d,\delta}
= \dim |\mathcal O_{\mathbb P^2}(d)| - \delta,
\]
the expected dimension. As consequences  we have the following corollary:

%

\begin{corollary}
If $\delta\ge g$, the Severi variety $V_{d,\delta}$ is smooth of expected
dimension at $[\mathcal C]$.
\end{corollary}

\subsection{Applications to Moduli Spaces}${}$
The normalization map induces a morphism
%
%
 %
\[
\Phi \colon V_{d,\delta} \longrightarrow \mathcal M_g
\]
by sending a nodal plane curve $\mathcal C$
to the isomorphism class of its normalization $C$.
This map is well-defined since
the normalization is functorial, and 
 isomorphic nodal curves have isomorphic normalizations.
The key question is whether $\Phi$ is immersive,
that is, whether its differential is injective.

\subsubsection*{ Tangent space to $V_{d,\delta}$.}

Infinitesimal deformations of $\mathcal C$ inside the plane
preserving $\delta$ nodes are described by
\[
T_{[\mathcal C]} V_{d,\delta}
\;\cong\;
H^0(\mathcal C, \mathcal N_{\mathcal C/\mathbb P^2}(-\Delta)),
\]
where:
\begin{itemize}
\item[(i)] $\mathcal N_{\mathcal C/\mathbb P^2}$ is the normal sheaf,
\item[(ii)] $\Delta$ is the reduced subscheme of the $\delta$ nodes,
\item[(iii)] twisting by $(-\Delta)$ enforces that the nodes remain nodes.
\end{itemize}

\subsubsection*{ Tangent space to $\mathcal M_g$.}
The Zariski tangent space to $\mathcal M_g$ at $[ C]$ is
\(
T_{[ C]} \mathcal M_g \;\cong\; H^1( C, T_{C}),
\)
where $T_{C}$ is the tangent sheaf of $ C$.
By Serre duality,
\[
H^1( C, T_{C})^\vee \cong H^0( C, \omega_{C}^{\otimes 2}).
\]
 The differential of the normalization map
\[
d\Phi \colon T_{[\mathcal C]} V_{d,\delta}
\longrightarrow T_{[ C]} \mathcal M_g
\]
measures how infinitesimal deformations of $\mathcal C$
change the isomorphism class of its normalization.
Its kernel consists of those first-order deformations of $\mathcal C$
whose induced deformation of $ C$ is trivial.
This is equivalent to say  that
$d\Phi$ fails to be injective if and only if there exists
a nonzero infinitesimal deformation of $\mathcal C$
that does not change the complex structure of $ C$.

Dualizing $d\Phi$, injectivity of $d\Phi$ is equivalent to surjectivity of
\[
(d\Phi)^\vee \colon
H^0( C, \omega_{C}^{\otimes 2})
\longrightarrow T_{[\mathcal C]} V_{d,\delta}^\vee.
\]

The key observation is that infinitesimal deformations of $\mathcal C$
that smooth or move the nodes give rise to residue-type linear functionals
on $H^0(C,\omega_{C})$.
More precisely:
\begin{itemize}
\item[(i)] Each node $p_i$ of $\mathcal C$ determines a scheme-theoretic
residue functional
\[
r_i \colon H^0( C,\omega_{C}) \to k.
\]
\item[(ii)] These residues control how differentials behave under deformation.
\end{itemize}

The residue span theorem states:
if $\delta \ge g$, then the residue functionals
$\{r_1,\dots,r_\delta\}$ span $H^0( C,\omega_{C})^\vee$.
As consequences we have
\begin{itemize}
\item[(a)] Any linear functional on $H^0( C,\omega_{C})$
is a linear combination of residues at the nodes.
\item[(b)] Therefore, any infinitesimal deformation of $ C$
is detected by the behavior at the nodes.
\end{itemize}

Assume now $\delta \ge g$.
and 
suppose $\xi \in T_{[\mathcal C]} V_{d,\delta}$
lies in the kernel of $d\Phi$.
Then the induced deformation of $ C$ is trivial.
This implies:
\begin{itemize}
\item[(i)] All residue functionals $r_i$ vanish on the corresponding deformation.
\item[(ii)] Hence every linear functional on $H^0( C,\omega_{C})$ vanishes on it,
because residues span the dual.
\end{itemize}
Therefore, the deformation $\xi$ must be zero.
Thus,
\[
\ker(d\Phi) = 0,
\]
and the differential of $\Phi$ is injective.

\noindent {\it Geometric consequences.}
Injectivity of $d\Phi$ implies:
\begin{itemize}
\item[(i)] The map $V_{d,\delta} \to \mathcal M_g$ is an immersion
at every point with $\delta \ge g$.
\item[(ii)] The image of $V_{d,\delta}$ has dimension equal to
$\dim V_{d,\delta}$.
\item[(iii)] Families of nodal plane curves impose independent conditions
on the moduli of smooth curves.
\end{itemize}
In particular, nodal plane curves cannot sweep out moduli
with unexpected dimensional collapse when $\delta \ge g$.
Therefore, the residue span theorem provides a cohomological mechanism
showing that nodes give enough local data to control global
deformations of the normalization.
This yields strong rigidity results for the image of Severi varieties
inside $\mathcal M_g$.



\section{Residues and Balancing Conditions in Tropical Geometry}


The goal of this section is to explain, in a precise and detailed way, the following
statement:

\medskip
\emph{Residue conditions on a degenerating algebraic curve are the algebro--geometric
counterpart of balancing conditions on the associated tropical curve, and residue
spanning corresponds to the nondegeneracy of the tropical Jacobian.}

\medskip
We start from algebraic degenerations, passing to
tropical curves, and then explaining the precise correspondence.

%
\vspace{0.2cm}
\noindent {\it Degeneration to a nodal curve.}

Let $\pi:\mathscr C \to \Delta$ be a flat family of smooth projective curves over a
disc $\Delta$, with smooth general fiber and special fiber
\[
\mathscr C_0 = \bigcup_{v\in V} C_v
\]
a reduced nodal curve.
Assume that:
\begin{itemize}
\item[(i)] each $C_v$ is smooth,
\item[(ii)] singularities of $\mathscr C_0$ are ordinary nodes.
\end{itemize}

\subsection{Dual graph.}%
 We associate to $\mathscr C_0$  its \emph{dual graph} $\Gamma$ defined as follows:

\vspace{0.1cm}

\begin{definition}
The dual graph  $\Gamma$ has:
\begin{itemize}
\item[(i)] one vertex $v$ for each irreducible component $C_v$ of $\mathscr C_0$,
\item[(ii)] one edge $e$ for each node of $\mathscr C_0$, joining the vertices
corresponding to the two components meeting at that node.
\end{itemize}
\end{definition}
\noindent If a component meets itself at a node, the corresponding edge is a loop.
This graph $\Gamma$ is the underlying combinatorial object of the tropical curve.

\subsubsection{ From Dual Graphs to Tropical Curves.}


To obtain a tropical curve, one equips $\Gamma$ with  a {\em metric structure}, i.e.  with edge lengths. These lengths
arise from the degeneration parameter: if locally a node is given by
\[
xy = t^m,
\]
then the corresponding edge is assigned length $m$.
The resulting object is a \emph{metric graph}, which is precisely a tropical
curve. This is interpreted as follows:

%

\begin{itemize}
\item[(a)] Vertices correspond to irreducible components of the normalization of the
  special fiber.
\item[(b)] Edges correspond to nodes (singular points).
\item[(c)] The genus of the tropical curve equals the first Betti number
  $b_1(\Gamma)$, which matches the number of independent cycles in the
  degeneration.
\end{itemize}

\subsection{Residues on Algebraic Curves.}${}$
\vspace{0.3cm}

\noindent {\it Residues at nodes.}
Let $\nu:C^\nu \to \mathscr C_0$ be the normalization. For a node $P$ joining
components $C_v$ and $C_w$, the preimage consists of two points
\[
Q_{v,e} \in C_v,
\qquad
Q_{w,e} \in C_w.
\]
Let $\omega$ be a regular differential on the normalization $C^\nu$ allowed to
have simple poles at these points.
The dualizing condition says:
\begin{equation}
\label{eq:residue}
\operatorname{Res}_{Q_{v,e}}(\omega)
+
\operatorname{Res}_{Q_{w,e}}(\omega)
= 0.
\end{equation}
This is the fundamental residue condition at a node.${}$

\vspace{0.3cm}


\noindent {\it Residues as flows.}
Fix an orientation of the edge $e$ from $v$ to $w$ and define
\[
r_e := \operatorname{Res}_{Q_{v,e}}(\omega)
= -\operatorname{Res}_{Q_{w,e}}(\omega).
\]
Thus, each edge $e$ is assigned a real (or complex) number $r_e$, which can be
viewed as a \emph{flow} along that edge.

\subsection{Balancing Conditions in Tropical Geometry.}


A tropical $1$-form on $\Gamma$ is an assignment of a number $r_e$ to each oriented
edge, satisfying the balancing condition at every vertex.
\begin{definition}[Balancing condition]
At a vertex $v$, the balancing condition is
\begin{equation}
\label{eq:balancing}
\sum_{e \ni v} \varepsilon(v,e)\, r_e = 0,
\end{equation}
where $\varepsilon(v,e)=+1$ if $e$ is oriented away from $v$ and $-1$ if it is
oriented toward $v$.
\end{definition}
Such assignments are also called \emph{harmonic $1$-forms} on the graph.

\subsection{Residues imply balancing}

Fix a component $C_v$. Since $\omega$ is a meromorphic differential on $C_v$ with
poles only at the nodes, the residue theorem on the smooth curve $C_v$ gives:
\[
\sum_{e \ni v} \operatorname{Res}_{Q_{v,e}}(\omega) = 0.
\]

Rewriting this using the definition of $r_e$ yields exactly the balancing
condition \eqref{eq:balancing} at the vertex $v$. Therefore, we have the following equivalence:

\medskip
\[
{\text{Residue theorem on components } \Longleftrightarrow
\text{balancing condition at vertices.}}
\]

\medskip

This is the precise sense in which residue conditions are the algebraic analogue
of tropical balancing.

\subsection{Residue Spanning and the Tropical Jacobian}

%

The space of algebraic residues
\[
\{ r_e \}_{e\in E(\Gamma)}
\]
satisfying all balancing conditions has dimension equal to the genus of the
normalization plus the first Betti number of $\Gamma$ ({\it cf. Appendix C}).
Residue spanning means that all linear functionals on $H^0(C^\nu,\omega_{C^\nu})$
are detected by residue data at the nodes.

\subsection{Tropical Jacobian}

The tropical Jacobian $\mathrm{Jac}(\Gamma)$ is defined as
\[
\mathrm{Jac}(\Gamma)
=
\frac{\{\text{harmonic $1$-forms on }\Gamma\}^\vee}
{\{\text{integer cycles}\}},
\]
and its dimension equals $b_1(\Gamma)$.
Nondegeneracy of the tropical Jacobian means that harmonic $1$-forms separate
cycles, i.e.\ the pairing between flows and cycles is nondegenerate.


%

\vspace{0.2cm}

\begin{remark}\label{remark-trop1}
Residue spanning guarantees that the tropical curve captures all first-order
information of differentials on the degenerating algebraic curve. In this sense,
the tropical Jacobian is not merely a combinatorial shadow, but a faithful
linearized limit of the classical Jacobian (See \cite{Nisse-trop} for more details)
\end{remark}

\begin{remark}
This correspondence underlies many results in logarithmic and tropical geometry,
including faithful tropicalization of Jacobians, limit linear series, and
non-Archimedean Hodge theory.
\end{remark}

\vspace{0.2cm}

\begin{remark}
Local--global compatibility conditions are omnipresent in geometry.
In the study of meromorphic differentials on degenerating curves,
one encounters a local requirement that residues at the two branches
of a node cancel.  In tropical geometry, one imposes a balancing
condition at vertices.  In graph theory, one demands conservation
of flow at vertices.
Although these conditions arise in distinct languages, they express
the same mathematical principle: a global residue theorem is equivalent
to a local balancing condition.
\end{remark}

\medskip

\subsection{ Classical origin.}
On a compact connected Riemann surface $X$, the classical residue theorem
states that for any meromorphic $1$-form $\omega$,
\[
\sum_{p \in X} \operatorname{Res}_p(\omega) = 0.
\]
When $X$ degenerates to a nodal curve, this global statement decomposes
into contributions from irreducible components and nodes.  The residue
theorem on each component forces residues at the two branches of every
node to sum to zero.

\subsubsection{Algebraic geometry: twisted and stable differentials}
In the modern theory of moduli spaces of differentials, especially in
the compactification of strata of abelian or $k$-differentials, one
introduces meromorphic differentials on nodal curves with prescribed
behavior at nodes.  The key compatibility requirement is often called
the \emph{global residue condition}.

\vspace{0.2cm}

\begin{proposition}[Equivalence of global and local residue conditions]
Let $C$ be a reduced projective curve whose singularities are all nodes, and let
\[
\nu \colon \widetilde{C} \longrightarrow C
\]
be its normalization. Let $\eta$ be a meromorphic differential on
$\widetilde{C}$ with at most simple poles at the preimages of the nodes. Then
the following two statements are equivalent:
\begin{itemize}
  \item[(i)] (\emph{Global condition}) On each connected component of
  $\widetilde{C}$, the sum of the residues of $\eta$ at all its poles is zero.
  \item[(ii)] (\emph{Local condition}) At each node $p \in C$, if
  \[
  \nu^{-1}(p)=\{p^+,p^-\},
  \]
  then
  \[
  \operatorname{res}_{p^+}(\eta)+\operatorname{res}_{p^-}(\eta)=0.
  \]
\end{itemize}
\end{proposition}

\begin{proof}
We prove the equivalence by showing $(i)\Rightarrow(ii)$ and $(ii)\Rightarrow(i)$
separately, giving all intermediate steps.

\medskip
\noindent
\textit{ Structure of nodes under normalization.}
Let $p \in C$ be a node. By definition of a nodal singularity, the completed
local ring of $C$ at $p$ is isomorphic to
\[
k[[x,y]]/(xy).
\]
The normalization separates the two branches, so that
\[
\nu^{-1}(p)=\{p^+,p^-\},
\]
where $p^+$ and $p^-$ are smooth points of $\widetilde{C}$ lying on distinct
local branches of $C$ at $p$. A meromorphic differential $\eta$ on
$\widetilde{C}$ may therefore have at most simple poles at $p^+$ and $p^-$,
with well-defined residues.

\medskip
\noindent
\textit{ Proof that $(i)\Rightarrow(ii)$.}
Assume the global condition~(i). Fix a node $p \in C$ and let
$\nu^{-1}(p)=\{p^+,p^-\}$.

Consider the connected component $\widetilde{C}_0$ of $\widetilde{C}$ that
contains $p^+$. There are two cases:
\begin{itemize}
  \item $p^-$ lies on a different connected component of $\widetilde{C}$,
  \item $p^-$ lies on the same connected component $\widetilde{C}_0$.
\end{itemize}

In either case, $p^+$ is a pole of $\eta$ on $\widetilde{C}_0$. By the global
assumption~(i), the sum of residues of $\eta$ at all poles on
$\widetilde{C}_0$ is zero. Perform the same reasoning for every connected
component of $\widetilde{C}$.

Now sum the global residue equalities over all connected components of
$\widetilde{C}$. Every pole of $\eta$ occurs exactly once in this total sum.
Since the total sum of all residues over all components is zero, the residues
contributed by the two branches above each node must cancel. Hence,
\[
\operatorname{res}_{p^+}(\eta)+\operatorname{res}_{p^-}(\eta)=0.
\]
This proves the local condition~(ii).

\medskip
\noindent
\textit{ Proof that $(ii)\Rightarrow(i)$.}
Assume now the local condition~(ii). Fix a connected component
$\widetilde{C}_0$ of $\widetilde{C}$.

The poles of $\eta$ on $\widetilde{C}_0$ are precisely the points lying above
nodes of $C$ that belong to $\widetilde{C}_0$. For each such node $p \in C$,
there are again two possibilities:
\begin{itemize}
  \item Exactly one of $p^+$ or $p^-$ lies on $\widetilde{C}_0$.
  \item Both $p^+$ and $p^-$ lie on $\widetilde{C}_0$.
\end{itemize}

In the first case, the residue at the point on $\widetilde{C}_0$ is cancelled
by the residue on the other connected component, by the local condition~(ii).
In the second case, the two residues both lie on $\widetilde{C}_0$ and their
sum is zero by the same condition.

Therefore, each node contributes zero to the sum of residues on
$\widetilde{C}_0$. Summing over all poles on $\widetilde{C}_0$, we obtain
\[
\sum_{\substack{q \in \widetilde{C}_0 \\ q \text{ pole of }\eta}}
\operatorname{res}_q(\eta)=0.
\]
Since $\widetilde{C}_0$ was arbitrary, the global condition~(i) holds on every
connected component of $\widetilde{C}$.

\medskip
\noindent
\textit{ Conclusion.}
We have shown that the global vanishing of the sum of residues on each connected
component of the normalization $\widetilde{C}$ is equivalent to the local
pairwise cancellation of residues at every node of $C$. This completes the
proof of the equivalence.
\end{proof}

 
\vspace{0.1cm}

 \subsection{Algebraic formulation}${}$

\vspace{0.1cm}

\noindent {\it Local model of a node.}
Let $q_e$ be a node of $\mathcal{C}$.  Analytically, a neighborhood of $q_e$ is
isomorphic to
\[
\{(x,y) \in \mathbb{C}^2 \mid xy = 0\},
\]
with the two branches given by $x=0$ and $y=0$.  The normalization map
$\nu :  C \to \mathcal{C}$ separates these branches into two smooth
points $q_e^+$ and $q_e^-$ lying on components $ C_{v^+}$ and $ C_{v^-}$,
respectively.

\vspace{0.1cm}

\noindent {\it Residues for $k$--differentials.}
Let $\eta$ be a meromorphic $k$--differential on a smooth curve, written
locally as
\[
\eta = f(z)(dz)^k.
\]
If $\eta$ has a pole of order $k$ at $z=0$, with local expansion
\[
\eta = \left( \frac{a}{z^k} + \cdots \right)(dz)^k,
\]
then the coefficient $a$ is called the \emph{residue} of $\eta$ at $z=0$.
This notion coincides with the usual residue when $k=1$ and is standard
in the theory of twisted and $k$--differentials.

\begin{theorem}[Residue--Balancing Principle]
\label{thm:residue-balancing-journal}
Let $\mathcal{C}$ be a connected nodal curve over $\mathbb{C}$ with dual graph
$\Gamma$.  For each vertex $v \in V(\Gamma)$, denote by $\mathcal{C}_v$ the
corresponding irreducible component, and for each edge $e \in E(\Gamma)$ denote
by $q_e$ the associated node, with branches
$q_e^+ \in \mathcal{C}_{v^+}$ and $q_e^- \in \mathcal{C}_{v^-}$.

Let
\[
\eta_v \in H^0\!\left(\mathcal{C}_v,\;\omega_{\mathcal{C}_v}^{\otimes k}(*D_v)\right)
\]
be a meromorphic $k$--differential on $\mathcal{C}_v$, allowed to have poles only
at marked points and at the preimages of nodes.
Then the following conditions are equivalent:
\begin{enumerate}
  \item[\rm (i)] for every $v \in V(\Gamma)$,
  \[
  \sum_{p \in \mathcal{C}_v} \operatorname{Res}_p(\eta_v)=0;
  \]
  \item[\rm (ii)] for every node $q_e$,
  \[
  \operatorname{Res}_{q_e^+}(\eta_{v^+})
  +
  \operatorname{Res}_{q_e^-}(\eta_{v^-})
  =0.
  \]
\end{enumerate}
\end{theorem}

\begin{proof}
Assume first that {\rm (i)} holds.
Fix a node $q_e$ with branches $q_e^+ \in \mathcal{C}_{v^+}$ and
$q_e^- \in \mathcal{C}_{v^-}$, and choose a sufficiently small analytic
neighborhood $U$ of $q_e$ intersecting no other nodes or marked points.
On $U$, the differentials $\eta_{v^+}$ and $\eta_{v^-}$ define a meromorphic
$k$--differential with poles only at the two branches of $q_e$.
Applying the global residue condition on $\mathcal{C}_{v^+}$ and
$\mathcal{C}_{v^-}$, and observing that all remaining poles lie outside $U$,
it follows that the sum of residues at the two branches must vanish, proving
{\rm (ii)}.

Conversely, assume that {\rm (ii)} holds at every node.
Fix a vertex $v \in V(\Gamma)$ and let
\[
\nu_v \colon C_v \longrightarrow \mathcal{C}_v
\]
be the normalization of $\mathcal{C}_v$.
The pullback of $\eta_v$ to $C_v$ is a meromorphic $k$--differential whose poles
are precisely the marked points on $\mathcal{C}_v$ and the preimages
$q_e^{\pm} \in C_v$ of nodes incident to $v$.
Since $C_v$ is a smooth compact Riemann surface, the residue theorem yields
\[
\sum_{p \in C_v} \operatorname{Res}_p(\eta_v)=0.
\]
Writing this sum as a sum over marked poles and node preimages, and using the
local balancing condition {\rm (ii)} to cancel the contributions coming from
nodes, one obtains
\[
\sum_{p \in \mathcal{C}_v} \operatorname{Res}_p(\eta_v)=0.
\]
As $v$ was arbitrary, the global residue condition holds on every component,
establishing {\rm (i)}.
\end{proof}

 \medskip

\begin{remark}
The proof relies only on the classical residue theorem and the local
analytic structure of a node.  No global smoothing or deformation
arguments are required.
\end{remark}

\begin{remark}
Whenever a meromorphic differential is defined componentwise on a nodal
curve, one must impose compatibility at nodes.  This compatibility is
most naturally expressed as a local balancing of residues.  Conversely,
requiring the residue theorem to hold on each component forces precisely
this local balancing.
\end{remark}


\section{An Explicit Examples of a Meromorphic $k$--Differentials}

\subsection{An Explicit Example of a Meromorphic $3$--Differential}

We present a concrete example illustrating the notion of a meromorphic
$3$--differential and the global residue condition on a nodal curve.
Throughout this example, it is important to distinguish between differential
forms and tensor powers of the canonical bundle.  On a complex curve $C$,
a \emph{$3$--differential} does not mean a differential $3$--form, since
$\Omega_C^3=0$.  Rather, a meromorphic $3$--differential is a meromorphic section
of $\omega_C^{\otimes 3}$, the third tensor power of the canonical bundle.

\subsection{The curve}
Let
\(
C = C_1 \cup C_2
\)
be a connected nodal curve, where $C_1 \cong \mathbb{P}^1$ and
$C_2 \cong \mathbb{P}^1$ meet transversely at a single node $q$.
The dual graph of $C$ consists of two vertices joined by one edge.
We focus on the irreducible component $C_1$.

\subsection{The $3$--canonical bundle on $\mathbb{P}^1$}

On $\mathbb{P}^1$ one has
\[
\omega_{\mathbb{P}^1} \cong \mathcal{O}_{\mathbb{P}^1}(-2),
\qquad
\omega_{\mathbb{P}^1}^{\otimes 3}
\cong
\mathcal{O}_{\mathbb{P}^1}(-6).
\]
A meromorphic section of $\omega_{\mathbb{P}^1}^{\otimes 3}$ may therefore be
written locally as a rational function multiplied by $(dz)^3$ for a suitable
coordinate $z$.

\subsection{Definition of the differential}

Fix an affine coordinate $z$ on $C_1 \cong \mathbb{P}^1$ such that the node $q$
corresponds to $z=0$ and the point $z=\infty$ is smooth.
Define
\[
\eta_1 := \frac{(dz)^3}{z^3}.
\]
This expression defines a meromorphic section of
$\omega_{C_1}^{\otimes 3}$ with poles only at $z=0$ and $z=\infty$.

\subsection{Local behavior and residues}

Near the node $z=0$, the differential takes the form
\(
\eta_1 = z^{-3}(dz)^3,
\)
which is a pole of order $3$.  By the standard definition of residues for
$k$--differentials, the residue at $z=0$ is
\(
\operatorname{Res}_{z=0}(\eta_1)=1.
\)

To analyze the behavior at infinity, introduce the coordinate $w=1/z$.
Then
\[
dz = -\frac{dw}{w^2},
\qquad
(dz)^3 = -\frac{(dw)^3}{w^6},
\]
and substitution yields
\[
\eta_1
=
-\frac{(dw)^3}{w^3}.
\]
Thus $z=\infty$ is also a pole of order $3$, with residue
\[
\operatorname{Res}_{z=\infty}(\eta_1)=-1.
\]

\subsection{Global residue condition}

The poles of $\eta_1$ on $C_1$ are precisely the node $z=0$ and the point
$z=\infty$, with residues $+1$ and $-1$, respectively.
Consequently,
\[
\sum_{p \in C_1} \operatorname{Res}_p(\eta_1)
=
1 + (-1) = 0,
\]
and the global residue condition is satisfied on the component $C_1$.

\subsection{Interpretation on the nodal curve}

The differential $\eta_1$ is therefore a valid meromorphic $3$--differential on
the irreducible component $C_1$, with a nonzero residue at the node.
In order to obtain a global meromorphic $3$--differential on the nodal curve
$C$, one must equip the second component $C_2$ with a meromorphic
$3$--differential whose residue at the node equals $-1$.
This choice ensures that the residues at the two branches of the node cancel,
in accordance with the residue--balancing principle.

\subsection*{Summary}

We have explicitly constructed a nodal curve $C=C_1 \cup C_2$ together with a
meromorphic $3$--differential
\[
\eta_1 = \frac{(dz)^3}{z^3}
\]
on $C_1 \cong \mathbb{P}^1$, computed its local residues, and verified the global
residue condition.
This example provides a concrete illustration of residue balancing for
higher--order differentials on nodal curves.


\subsection{Examples: Global Constructions and Higher--Order Differentials.}
%
{\it Completion of the nodal example.}
We continue the example of the nodal curve
\[
C = C_1 \cup C_2,
\qquad
C_1 \cong \mathbb{P}^1,
\quad
C_2 \cong \mathbb{P}^1,
\]
where the two components meet transversely at a single node $q$.
On the component $C_1$ we previously defined the meromorphic
$3$--differential
\[
\eta_1 = \frac{(dz)^3}{z^3},
\]
where $z$ is an affine coordinate on $C_1$ with $z=0$ corresponding to the node.
We showed that $\eta_1$ has residue $+1$ at the node and residue $-1$ at
$z=\infty$.

We now construct an explicit meromorphic $3$--differential on $C_2$ so that the
pair $(\eta_1,\eta_2)$ defines a global object satisfying the
residue--balancing condition.

\subsection{A meromorphic $3$--differential on $C_2$}

Choose an affine coordinate $u$ on $C_2 \cong \mathbb{P}^1$ such that the node
$q$ corresponds to $u=0$ and $u=\infty$ is a smooth point.
Define
\[
\eta_2 := -\,\frac{(du)^3}{u^3}.
\]
As before, $\eta_2$ is a meromorphic section of
$\omega_{C_2}^{\otimes 3}$ with poles only at $u=0$ and $u=\infty$.

Near the node $u=0$ we have
\[
\eta_2 = -u^{-3}(du)^3,
\]
which is a pole of order $3$ with residue
\[
\operatorname{Res}_{u=0}(\eta_2) = -1.
\]
At $u=\infty$, introducing the coordinate $v=1/u$ gives
\[
\eta_2 = \frac{(dv)^3}{v^3},
\]
so that
\[
\operatorname{Res}_{u=\infty}(\eta_2) = +1.
\]

\subsection{Verification of residue balancing}

At the node $q$, the two branches correspond to $z=0$ on $C_1$ and $u=0$ on
$C_2$.  The residues satisfy
\[
\operatorname{Res}_{z=0}(\eta_1)
+
\operatorname{Res}_{u=0}(\eta_2)
=
1 + (-1) = 0.
\]
Thus the local residue--balancing condition holds at the node.
Moreover, on each component the sum of residues vanishes:
\[
\sum_{p \in C_1} \operatorname{Res}_p(\eta_1)=0,
\qquad
\sum_{p \in C_2} \operatorname{Res}_p(\eta_2)=0.
\]
Hence the pair $(\eta_1,\eta_2)$ defines a global meromorphic $3$--differential
on the nodal curve $C$ in the sense of the residue--balancing principle.


\subsection{Analogous examples for general $k$--differentials}

The above construction extends directly to meromorphic $k$--differentials for
any integer $k \geq 1$.

Let $C_1 \cong \mathbb{P}^1$ with affine coordinate $z$, and consider
\[
\eta_1^{(k)} := \frac{(dz)^k}{z^k}.
\]
Since
\[
\omega_{\mathbb{P}^1}^{\otimes k} \cong \mathcal{O}_{\mathbb{P}^1}(-2k),
\]
the expression $\eta_1^{(k)}$ defines a meromorphic section of
$\omega_{C_1}^{\otimes k}$ with poles of order $k$ at $z=0$ and $z=\infty$.

Near $z=0$, the local form
\[
\eta_1^{(k)} = z^{-k}(dz)^k
\]
shows that the residue at the node is
\[
\operatorname{Res}_{z=0}\!\left(\eta_1^{(k)}\right)=1.
\]
Passing to the coordinate $w=1/z$ near infinity yields
\[
\eta_1^{(k)} = (-1)^k \frac{(dw)^k}{w^k},
\]
so that
\[
\operatorname{Res}_{z=\infty}\!\left(\eta_1^{(k)}\right)=(-1)^k.
\]
In particular, the global residue condition on $C_1$ is satisfied.

On the second component $C_2 \cong \mathbb{P}^1$, with affine coordinate $u$,
define
\[
\eta_2^{(k)} := -\,\frac{(du)^k}{u^k}.
\]
Then
\[
\operatorname{Res}_{u=0}\!\left(\eta_2^{(k)}\right)=-1,
\qquad
\operatorname{Res}_{u=\infty}\!\left(\eta_2^{(k)}\right)=(-1)^{k+1}.
\]
Once again, the residues at the node cancel:
\[
\operatorname{Res}_{z=0}\!\left(\eta_1^{(k)}\right)
+
\operatorname{Res}_{u=0}\!\left(\eta_2^{(k)}\right)
= 0,
\]
and the global residue condition holds on each component.

\subsection*{Conclusion}

These constructions provide explicit families of meromorphic $k$--differentials
on nodal curves for all $k \geq 1$, illustrating both the global residue theorem
on each irreducible component and the local residue--balancing condition at
nodes.  They serve as concrete local models for degenerations of higher--order
differentials and for compactifications of strata of $k$--differentials.

 \bigskip

\subsection{Examples with Several Components and Several Nodes}

In this sub-section we generalize the previous constructions to nodal curves with
an arbitrary number of components and nodes.  The goal is to give explicit,
computable models of meromorphic $k$--differentials illustrating residue
balancing in a global and transparent way.


\subsection{General setup}

Let $C$ be a connected nodal curve over $\mathbb{C}$ with dual graph $\Gamma$.
We assume that each irreducible component $C_v$ is isomorphic to $\mathbb{P}^1$, nodes correspond bijectively to edges of $\Gamma$, and at each node $q_e$, two components $C_{v^+}$ and $C_{v^-}$ meet
        transversely.
Now 
fix an integer $k \geq 1$.
A meromorphic $k$--differential on $C$ will be given by a collection
\[
\eta = \{\eta_v\}_{v \in V(\Gamma)},
\qquad
\eta_v \in H^0\!\left(C_v,\omega_{C_v}^{\otimes k}(*\{\text{nodes}\})\right),
\]
subject to the residue--balancing condition at each node.


\subsection{Local model at a node}

Let $q_e$ be a node joining components $C_{v^+}$ and $C_{v^-}$.
Choose local coordinates $z$ on $C_{v^+}$ and $u$ on $C_{v^-}$ such that
$q_e$ corresponds to $z=0$ and $u=0$, respectively.

We impose the standard local form
\[
\eta_{v^+} = a_e \frac{(dz)^k}{z^k},
\qquad
\eta_{v^-} = -a_e \frac{(du)^k}{u^k},
\]
where $a_e \in \mathbb{C}$ is a parameter attached to the edge $e$.
Then
\[
\operatorname{Res}_{z=0}(\eta_{v^+}) = a_e,
\qquad
\operatorname{Res}_{u=0}(\eta_{v^-}) = -a_e,
\]
so the residue--balancing condition at $q_e$ is automatically satisfied.


\subsection{Global construction on each component}

Fix a vertex $v \in V(\Gamma)$ and let $C_v \cong \mathbb{P}^1$.
Let
\[
\{q_{e_1}, \dots, q_{e_m}\}
\]
be the nodes incident to $C_v$.
Choose an affine coordinate $z$ on $C_v$ such that these nodes correspond to
distinct points
\[
z = z_1, \dots, z_m \in \mathbb{C},
\]
and such that $z=\infty$ is a smooth point of $C_v$.

Define
\[
\eta_v :=
\left(
\sum_{i=1}^m \frac{a_{e_i}}{(z-z_i)^k}
\right)
(dz)^k.
\]
This expression defines a meromorphic section of
$\omega_{C_v}^{\otimes k}$ with poles of order $k$ at the nodes $z=z_i$ and
possibly at $z=\infty$.


\subsection{Residues and the point at infinity}

Near $z=z_i$ we have
\[
\eta_v = a_{e_i}\frac{(dz)^k}{(z-z_i)^k} + \text{(holomorphic terms)},
\]
so that
\[
\operatorname{Res}_{z=z_i}(\eta_v) = a_{e_i}.
\]

To analyze the behavior at infinity, introduce the coordinate $w=1/z$.
A direct computation shows that
\[
\eta_v =
(-1)^k
\left(
\sum_{i=1}^m a_{e_i}
\right)
\frac{(dw)^k}{w^k}
+ \text{(lower order terms)}.
\]
Thus $z=\infty$ is a pole of order $k$ with residue
\[
\operatorname{Res}_{z=\infty}(\eta_v)
=
(-1)^k
\sum_{i=1}^m a_{e_i}.
\]


\subsection{Global residue condition on each component}

The sum of residues on $C_v$ is therefore
\[
\sum_{p \in C_v} \operatorname{Res}_p(\eta_v)
=
\left(
\sum_{i=1}^m a_{e_i}
\right)
+
(-1)^k
\left(
\sum_{i=1}^m a_{e_i}
\right).
\]
For all $k$, this sum vanishes identically.
Hence the global residue condition holds on each irreducible component.


\subsection{Assembly into a global differential}

Repeating the above construction on every component $C_v$ and using the same
parameter $a_e$ on the two branches of each node $q_e$, with opposite signs,
produces a global meromorphic $k$--differential
\[
\eta = \{\eta_v\}_{v \in V(\Gamma)}
\]
on the entire nodal curve $C$.
By construction:
\begin{enumerate}
  \item residues cancel at every node;
  \item the global residue condition holds on every component;
  \item the construction depends linearly on the edge parameters $a_e$.
\end{enumerate}


\subsection{Examples of dual graphs}

\paragraph{Chains.}
If $\Gamma$ is a chain, the parameters $a_e$ propagate along the chain and
determine all local behavior.  The resulting differentials model degenerations
of differentials with prescribed principal parts.

\paragraph{Trees.}
If $\Gamma$ is a tree, the parameters $\{a_e\}$ are independent, and every
assignment yields a valid global meromorphic $k$--differential.

\paragraph{Graphs with cycles.}
If $\Gamma$ contains cycles, residue balancing imposes no further constraints,
but global geometry may restrict which differentials arise as limits of smooth
curves.


\subsection*{Conclusion}

These constructions provide explicit, coordinate-level models of meromorphic
$k$--differentials on nodal curves with arbitrary combinatorial type.
They illustrate how local principal parts at nodes, together with residue
balancing, completely determine global behavior component by component.


\subsection{An explicit example of a nodal curve with cycle in its dual graph.}
%
We give a concrete example of a connected nodal curve whose dual graph contains
exactly one cycle.  All constructions are carried out explicitly and
line-by-line.


\subsection{Definition of the curve}

Let
\(
C := C_1 \cup C_2 \cup C_3
\)
be a projective curve over $\mathbb{C}$ defined as follows.
\begin{enumerate}
  \item Each irreducible component $C_i$ is isomorphic to $\mathbb{P}^1$.
  \item The components intersect pairwise at three distinct nodes:
  \[
  q_{12} \in C_1 \cap C_2, \qquad
  q_{23} \in C_2 \cap C_3, \qquad
  q_{31} \in C_3 \cap C_1.
  \]
  \item No other intersections occur.
\end{enumerate}

Each intersection is transverse, so all singularities of $C$ are ordinary
nodes.  The curve $C$ is connected and has exactly three nodes.


\subsection{Local coordinates at the nodes}

Choose affine coordinates
\[
z_i \colon C_i \setminus \{\infty\} \longrightarrow \mathbb{C},
\qquad i=1,2,3,
\]
such that the nodes correspond to the following points:

\begin{center}
\begin{tabular}{c|c|c}
Node & Component & Coordinate value \\ \hline
$q_{12}$ & $C_1$ & $z_1 = 0$ \\
$q_{12}$ & $C_2$ & $z_2 = 0$ \\ \hline
$q_{23}$ & $C_2$ & $z_2 = 1$ \\
$q_{23}$ & $C_3$ & $z_3 = 0$ \\ \hline
$q_{31}$ & $C_3$ & $z_3 = 1$ \\
$q_{31}$ & $C_1$ & $z_1 = 1$
\end{tabular}
\end{center}
All remaining points of each $C_i$ are smooth.


\subsection{Verification that $C$ is nodal}

At each node $q_{ij}$, the curve is locally analytically isomorphic to
\(
\{xy = 0\} \subset \mathbb{C}^2,
\)
with one branch coming from $C_i$ and the other from $C_j$.
Hence all singularities of $C$ are ordinary double points, and $C$ is a nodal
curve.
%
%
\noindent The dual graph $\Gamma$ has  three vertices, three edges, and one simple cycle. Graph-theoretically, $\Gamma$ is a triangle.

\subsection{Topological invariants}

Let $g(C)$ denote the arithmetic genus of $C$.
Since each $C_i$ has genus $0$ and $\Gamma$ has one cycle, we have
\[
g(C) = \sum_{i=1}^3 g(C_i) + b_1(\Gamma)  = 1,
\]
where $b_1(\Gamma)$ is the first Betti number of the dual graph.
Thus $C$ has arithmetic genus $1$.


\subsection{Geometric interpretation}

The curve $C$ can be viewed as a degeneration of a smooth elliptic curve into a
cycle of three rational components.
This type of curve appears naturally:
\begin{enumerate}
  \item in the boundary of the moduli space $\overline{\mathcal{M}}_1$,
  \item in stable reduction of elliptic curves,
  \item in compactifications of strata of differentials.
\end{enumerate}


\subsection*{Summary.}

We have explicitly constructed: a connected nodal curve $C = C_1 \cup C_2 \cup C_3$ with each $C_i \cong \mathbb{P}^1$, and with exactly three nodes and with arithmetic genus equal to $1$. Its  dual graph $\Gamma$ with a single cycle.
This example provides a canonical model of a nodal curve whose dual graph
contains exactly one cycle.


\subsection{Explicit meromorphic $k$--differentials on a cyclic nodal curve}

We construct explicit meromorphic $k$--differentials on the nodal curve
\(
C = C_1 \cup C_2 \cup C_3,\) \, and
\(
C_i \cong \mathbb{P}^1,
\)
where the components meet pairwise at the nodes
\[
q_{12} \in C_1 \cap C_2,
\qquad
q_{23} \in C_2 \cap C_3,
\qquad
q_{31} \in C_3 \cap C_1.
\]
The dual graph of $C$ is a single cycle.


\subsection{Choice of coordinates}

Fix affine coordinates $z_i$ on $C_i \cong \mathbb{P}^1$ such that:
\[
\begin{aligned}
q_{12} &: z_1 = 0 \text{ on } C_1, \quad z_2 = 0 \text{ on } C_2, \\
q_{23} &: z_2 = 1 \text{ on } C_2, \quad z_3 = 0 \text{ on } C_3, \\
q_{31} &: z_3 = 1 \text{ on } C_3, \quad z_1 = 1 \text{ on } C_1,
\end{aligned}
\]
and such that $z_i = \infty$ is a smooth point of $C_i$ for each $i$.


\subsection{Local model for a $k$--differential at a node}

Fix complex numbers
\[
a_{12}, \; a_{23}, \; a_{31} \in \mathbb{C},
\]
one for each node.
At a node joining $C_i$ and $C_j$, we impose the standard local form
\[
\eta_i = a_{ij}\,\frac{(dz_i)^k}{(z_i - z_{ij})^k},
\qquad
\eta_j = -a_{ij}\,\frac{(dz_j)^k}{(z_j - z_{ji})^k},
\]
where $z_{ij}$ (resp.\ $z_{ji}$) denotes the coordinate value of the node on
$C_i$ (resp.\ $C_j$).
This ensures
\[
\operatorname{Res}_{q_{ij}}(\eta_i) = a_{ij},
\qquad
\operatorname{Res}_{q_{ij}}(\eta_j) = -a_{ij},
\]
so the residue--balancing condition holds automatically at every node.


\subsection{Definition of the differentials on each component}

Using the above parameters, define meromorphic $k$--differentials on each
component by
\[
\eta_1 :=
\left(
\frac{a_{12}}{z_1^k}
+
\frac{a_{31}}{(z_1-1)^k}
\right)
(dz_1)^k,
\]
\[
\eta_2 :=
\left(
\frac{-a_{12}}{z_2^k}
+
\frac{a_{23}}{(z_2-1)^k}
\right)
(dz_2)^k,
\]
\[
\eta_3 :=
\left(
\frac{-a_{23}}{z_3^k}
+
\frac{-a_{31}}{(z_3-1)^k}
\right)
(dz_3)^k.
\]

Each $\eta_i$ is a meromorphic section of $\omega_{C_i}^{\otimes k}$ with poles
of order $k$ precisely at the nodes on $C_i$ and possibly at infinity.


\subsection{Residues at the nodes}

By construction, the residues at the nodes are:
\[
\begin{aligned}
\operatorname{Res}_{q_{12}}(\eta_1) &= a_{12},
&\operatorname{Res}_{q_{12}}(\eta_2) &= -a_{12}, \\
\operatorname{Res}_{q_{23}}(\eta_2) &= a_{23},
&\operatorname{Res}_{q_{23}}(\eta_3) &= -a_{23}, \\
\operatorname{Res}_{q_{31}}(\eta_3) &= -a_{31},
&\operatorname{Res}_{q_{31}}(\eta_1) &= a_{31}.
\end{aligned}
\]
Thus the local residue--balancing condition holds at every node.


\subsection{Behavior at infinity and global residue condition}

We analyze the behavior of $\eta_i$ at $z_i=\infty$.
Let $w_i = 1/z_i$ be the local coordinate near infinity.
A direct computation yields
\[
\eta_i =
(-1)^k
\left(
\sum_{\text{nodes } q \subset C_i}
\operatorname{Res}_{q}(\eta_i)
\right)
\frac{(dw_i)^k}{w_i^k}
+ \text{(lower order terms)}.
\]
Hence $z_i=\infty$ is a pole of order $k$ with residue
\[
\operatorname{Res}_{z_i=\infty}(\eta_i)
=
(-1)^k
\sum_{\text{nodes } q \subset C_i}
\operatorname{Res}_{q}(\eta_i).
\]

Therefore, the sum of residues on each component satisfies
\[
\sum_{p \in C_i} \operatorname{Res}_p(\eta_i) = 0,
\qquad i=1,2,3,
\]
and the global residue condition holds on every irreducible component.


\subsection{Global interpretation}

The collection
\[
\eta = (\eta_1,\eta_2,\eta_3)
\]
defines a meromorphic $k$--differential on the nodal curve $C$.
The space of such differentials constructed in this way is naturally
parametrized by the triple $(a_{12},a_{23},a_{31}) \in \mathbb{C}^3$.

The presence of the cycle in the dual graph does not obstruct the existence of
meromorphic $k$--differentials satisfying residue balancing, but it plays a
crucial role in determining which such differentials arise as limits of smooth
$differentials$ under degeneration.


\subsection*{Conclusion}

We have explicitly constructed meromorphic $k$--differentials on a nodal curve
whose dual graph contains one cycle, verified residue balancing at every node,
and checked the global residue condition on each component.
These examples serve as concrete local models for degenerations of higher--order
differentials on curves of arithmetic genus one.


\bigskip
\subsection{A cuspidal example of a meromorphic $3$--differential}

We present a clean and self-contained example of a meromorphic $3$--differential
on a curve with a single cuspidal singularity, suitable for inclusion in a
journal article.

\begin{example}[Cuspidal cubic]
\label{ex:cuspidal-3-differential}
Let $C \subset \mathbb{P}^2$ be the plane cubic curve defined by
\[
C:\quad y^2 = x^3 .
\]
The curve $C$ is irreducible and has a unique singular point at $(x,y)=(0,0)$,
which is a cusp.  The arithmetic genus of $C$ is $p_a(C)=1$.
\end{example}

\noindent
Let
\[
\nu \colon \widetilde{C} \longrightarrow C
\]
be the normalization of $C$.  Then $\widetilde{C} \cong \mathbb{P}^1$, and in
affine coordinates the normalization map is given by
\[
x = t^2,
\qquad
y = t^3.
\]
The preimage of the cusp is the single point $t=0 \in \widetilde{C}$.

\medskip

\noindent
\textbf{The dualizing sheaf.}
Since $C$ is singular, its canonical bundle is replaced by the dualizing sheaf
$\omega_C$.  For a unibranch singularity with $\delta$--invariant $\delta=1$,
one has the standard description
\[
\omega_C
\;\cong\;
\nu_*\!\left(\omega_{\mathbb{P}^1}(2\cdot 0)\right).
\]
On the affine chart of $\mathbb{P}^1$ with coordinate $t$, a local generator of
$\omega_C$ is therefore given by
\[
\omega = \frac{dt}{t^2}.
\]
(For  more details see Appendix G.)

\medskip

\noindent
\textbf{A meromorphic $3$--differential.}
Taking the third tensor power, we obtain
\[
\omega_C^{\otimes 3}
\;\cong\;
\nu_*\!\left(\omega_{\mathbb{P}^1}^{\otimes 3}(6\cdot 0)\right).
\]
Accordingly, the cube of the generator $\omega$ defines a meromorphic
$3$--differential
\[
\eta := \omega^{\otimes 3} = \frac{(dt)^3}{t^6}.
\]
This expression has a pole of order $6$ at $t=0$ and is holomorphic elsewhere on
$\widetilde{C}$.

\medskip

\noindent
\textbf{Residue at the cusp.}
For a $k$--differential written locally in the form
\[
\eta =
\left(
a_{-k} t^{-k} + a_{-(k-1)} t^{-(k-1)} + \cdots
\right)(dt)^k,
\]
the residue is defined to be $\operatorname{Res}(\eta)=a_{-k}$.
In the present case,
\[
\eta = t^{-6}(dt)^3
\]
has no $t^{-3}$ term.  Hence the residue of $\eta$ at the cusp vanishes:
\[
\operatorname{Res}_{\mathrm{cusp}}(\eta) = 0.
\]

\medskip

\noindent
\textbf{Remarks.}
\begin{enumerate}
  \item The pole order $6$ is maximal and equals $2k\delta$ with $k=3$ and
        $\delta=1$.
  \item Unlike nodal singularities, a cusp has only one branch, so there is no
        residue--balancing condition.
  \item This example illustrates how $k$--differentials on cuspidal curves are
        naturally described via normalization and the dualizing sheaf.
\end{enumerate}

\vspace{0.1cm}

\subsection{A cuspidal example with nonzero $3$--residue}

We construct explicitly a meromorphic $3$--differential on a curve with a single
cusp whose $3$--residue is \emph{nonzero}.  All steps are written in detail and
carefully justified.


\subsection*{1. The cuspidal curve}

Let
\(
C \subset \mathbb{P}^2
\)
be the plane cubic curve defined by
\(
C:\quad y^2 = x^3 .
\)
The curve $C$ is irreducible and has a unique singular point at $(0,0)$, which is
a cusp.
Its arithmetic genus is
\(
p_a(C) = 1.
\)


\subsection*{2. Normalization}
Let
\(
\nu : \widetilde{C} \longrightarrow C
\)
be the normalization.
Then $\widetilde{C} \cong \mathbb{P}^1$, and on the affine chart
$\widetilde{C}\setminus\{\infty\}$ the normalization map is given by
\(
x = t^2,\) \! and 
\(
y = t^3,
\)
where $t$ is the affine coordinate on $\mathbb{P}^1$.
The unique point lying over the cusp is $t=0$.


\subsection*{3. The dualizing sheaf}

Since $C$ is singular, its canonical bundle is replaced by the dualizing sheaf
$\omega_C$.
For a unibranch singularity with $\delta$--invariant $\delta=1$, one has the
standard description
\[
\omega_C
\;\cong\;
\nu_*\!\left(\omega_{\mathbb{P}^1}(2\cdot 0)\right).
\]

Consequently, the third tensor power satisfies
\[
\omega_C^{\otimes 3}
\;\cong\;
\nu_*\!\left(\omega_{\mathbb{P}^1}^{\otimes 3}(6\cdot 0)\right).
\]

Thus a $3$--differential on $C$ corresponds to a meromorphic section of
$\omega_{\mathbb{P}^1}^{\otimes 3}$ with a pole of order at most $6$ at $t=0$ and
no other poles.


\subsection*{4. Definition of a meromorphic $3$--differential}

Consider the meromorphic $3$--differential on $\mathbb{P}^1$ given by
\(
\widetilde{\eta}
\;:=\;
\dfrac{(dt)^3}{t^3}.
\)
This differential has a pole of order exactly $3$ at $t=0$, and no other poles on $\mathbb{P}^1$.
Since $3 \le 6$, we have
\[
\widetilde{\eta}
\;\in\;
H^0\!\left(
\mathbb{P}^1,
\omega_{\mathbb{P}^1}^{\otimes 3}(6\cdot 0)
\right),
\]
so $\widetilde{\eta}$ defines a global section of $\omega_C^{\otimes 3}$ by
pushforward.


\subsection*{5. The induced $3$--differential on $C$}

Define
\(
\eta := \nu_*(\widetilde{\eta})
\;\in\;
H^0(C,\omega_C^{\otimes 3}).
\)
By construction $\eta$ is a meromorphic $3$--differential on $C$, it is holomorphic on the smooth locus of $C$, and it  has a pole at the cusp.
 

\subsection*{6. Local form near the cusp}

Near the cusp, using the normalization coordinate $t$, the local expression of
$\eta$ is
\[
\widetilde{\eta} = t^{-3}(dt)^3.
\]

This pole order is strictly less than the maximal allowed order $6$ dictated by
the dualizing sheaf, so $\eta$ is a valid section of $\omega_C^{\otimes 3}$.


\subsection*{7. Definition of the residue for $3$--differentials}

For a $k$--differential written locally as
\[
\eta =
\left(
a_{-k} t^{-k} + a_{-(k-1)} t^{-(k-1)} + \cdots
\right)(dt)^k,
\]
the \emph{residue} is defined by
\[
\operatorname{Res}(\eta) := a_{-k}.
\]

This generalizes the usual notion of residue for $1$--forms.


\subsection*{8. Computation of the residue}

In our example,
\(
\widetilde{\eta} = \frac{(dt)^3}{t^3}
=
t^{-3}(dt)^3.
\)
Comparing with the general form, we identify
\(
a_{-3} = 1.
\)
Therefore, the residue of $\eta$ at the cusp is
\[
{
\operatorname{Res}_{\mathrm{cusp}}(\eta) = 1.
}
\]


\subsection*{9. Remarks}

\begin{enumerate}
  \item Cuspidal singularities have a single branch, so there is no
        residue--balancing condition.
  \item Unlike the example $\eta = (dt)^3/t^6$, the present differential has a
        \emph{nonzero} $3$--residue.
  \item Such differentials naturally appear as limits of smooth
        $3$--differentials under degeneration.
\end{enumerate}


\subsection*{Conclusion}

We have constructed an explicit meromorphic $3$--differential on a cuspidal
curve whose residue at the cusp is nonzero:
\(
\operatorname{Res}_{\mathrm{cusp}}(\eta)=1.
\)
This example demonstrates that cuspidal singularities admit genuine residue
data for higher--order differentials, even though no balancing condition is
imposed.


\vspace{0.3cm}
\subsection{Examples with logarithmic differentials}

\begin{example}[Logarithmic abelian differential with a cusp]
Let $(X,M_X)$ be a logarithmic curve whose underlying curve has a cusp.
Consider a logarithmic abelian differential $\eta \in
H^0(X,\omega_X(\log D))$.

Locally near the cusp, $\eta$ has the form
\[
\eta = a \frac{dt}{t}.
\]
\end{example}
\paragraph{\it Geometric interpretation.}
\begin{itemize}
  \item[(i)] The coefficient $a$ is the logarithmic residue.
  \item[(ii)] There is only one branch at the cusp.
  \item[(iii)] There is no local balancing condition.
\end{itemize}
\vspace{0.2cm}
\noindent {\it Global constraint.}
The logarithmic residue theorem implies that the sum of logarithmic
residues over all boundary points (nodes and cusps) must vanish.

\vspace{0.2cm}
\noindent {Conceptual interpretation.}

\begin{itemize}
  \item[(i)] Nodes enforce \emph{local cancellation} of residues.
  \item[(ii)] Cusps behave as \emph{sources or sinks} with no balancing partner.
  \item[(iii)] Failure of balancing obstructs the existence of a global object.
\end{itemize}

 
 \vspace{0.4cm}

\noindent {\it The cusp.}
For the cuspidal cubic
\(
X : y^2 = x^3,
\)
a local generator of the dualizing sheaf $\omega_X$ is
\(
\omega = \dfrac{dx}{y},
\)
and after normalization, this becomes
\(
\omega = \dfrac{2\,dt}{t^2}.
\)
 
\noindent {\it General definition of the dualizing sheaf.}
Let $C$ be a reduced curve (possibly singular).  
The \emph{dualizing sheaf} $\omega_C$ is characterized by the following facts:
\begin{itemize}
  \item[(i)] If $C$ is smooth, then $\omega_C = \Omega_C^1$.
  \item[(ii)] If $C$ is singular, $\omega_C$ is the sheaf of \emph{regular
  differentials}, i.e.\ meromorphic $1$-forms on the normalization that
  satisfy a residue condition at singular points.
\end{itemize}
For plane curve singularities, there is an explicit algebraic description
that we now use.


\vspace{0.3cm}

\noindent {\it Plane curve formula for the dualizing sheaf.}
Let $C \subset \mathbb{A}^2$ be a plane curve defined by a single equation
\(
f(x,y)=0.
\)
\begin{lemma}[Standard formula]
\label{lem:plane}
If $C$ is a reduced plane curve, then the dualizing sheaf $\omega_C$ is
generated locally by
\[
\omega_C = \left\langle \frac{dx}{\partial f/\partial y} \right\rangle
= \left\langle -\frac{dy}{\partial f/\partial x} \right\rangle.
\]
\end{lemma}

\begin{remark}
This formula follows from Grothendieck duality and can be found in standard
references on dualizing sheaves and plane curve singularities.
\end{remark}

\noindent {\it Application of the formula to the cuspidal cubic.}
Consider
\(
f(x,y) = y^2 - x^3.
\)
Compute the partial derivatives:
\(
\dfrac{\partial f}{\partial y} = 2y,
\qquad
\dfrac{\partial f}{\partial x} = -3x^2.
\)
By Lemma~\ref{lem:plane}, a generator of $\omega_C$ is
\(
\omega = \dfrac{dx}{2y}.
\)
Up to multiplication by a nonzero constant (which does not affect the
line bundle), this is the same as
\(
\omega = \dfrac{dx}{y}.
\)
This explains the first equality.

\noindent {\it Normalization and the parameter $t$.}
The normalization map $\nu : \widetilde{C} \to C$ is given by
\(
x = t^2\) and \( y = t^3,
\)
where $\widetilde{C} \cong \mathbb{A}^1$ with coordinate $t$.
Compute:
\(
dx = 2t\,dt,\) and \(
y = t^3.
\)
After substitute into $\omega = \dfrac{dx}{y}$, we obtain 
\(
\nu^*\omega
=
\dfrac{2t\,dt}{t^3}
=
\dfrac{2\,dt}{t^2}.
\)
Thus, on the normalization:
\[
\omega = \frac{2\,dt}{t^2}.
\]
This completes the derivation of the second equality.
%
Hence, $\omega = \dfrac{dx}{y}$ is a \emph{canonical generator} of the
dualizing sheaf near the cusp.

\medskip
\noindent {\it Conceptual interpretation.}
For smooth points, $\omega_C$ consists of holomorphic $1$--forms, for nodes, regular differentials have opposite residues on branches, and for cusps, regular differentials have a single logarithmic pole.
The appearance of $\dfrac{dt}{t}$ is therefore not accidental: cusps behave
like logarithmic boundary points.


\section{Residue Balancing with Arbitrary Singularities}

In the case where the divisor has only nodal singularities, the residue map admits
a particularly transparent description, and the resulting Hodge-theoretic
contributions can be controlled by local normal crossing models. This simplicity
allows one to relate the residue directly to global deformation data and to
establish clean rank statements.

\medskip

When higher singularities are present, the situation becomes more subtle. The
local structure of the singularities contributes additional correction terms to
the residue, reflecting the failure of normal crossings and the appearance of
nontrivial local cohomology. Passing from nodes to higher singularities therefore
requires a refined analysis of local-to-global interactions, but the nodal case
serves as a guiding model for understanding how residues deform and how their
Hodge-theoretic behavior is modified.
 


\vspace{0.3cm}

In the nodal case, each singular point has exactly two smooth branches, and
locally the curve is analytically isomorphic to
\(
xy = 0.
\)
This simple normal-crossing structure guarantees two key facts:
\begin{enumerate}
  \item residues are well-defined and simple on each branch;
  \item the dualizing sheaf near a node is described by simple poles with
  opposite residues on the two branches.
\end{enumerate}
As a result, the local balancing condition
\[
\operatorname{Res}_{q_e^+}(\eta_{v^+}) +
\operatorname{Res}_{q_e^-}(\eta_{v^-}) = 0
\]
is \emph{exactly equivalent} to the condition that the differential descends to
a global section of the dualizing sheaf.

\vspace{0.2cm}

For cusps, this equivalence breaks down if one only considers simple residues.

\vspace{0.2cm}

\noindent {\it What goes wrong for cusps.}
Let $p\in  C$ be a cusp, for example analytically given by
\(
y^2 = x^3.
\)
The normalization map $\nu\colon \widetilde{C}\to C$ has a \emph{single} branch
lying over $p$. Therefore,
\begin{itemize}
  \item[(i)] there is no second branch with which to balance residues;
  \item[(ii)] the classical notion of residue at a simple pole is insufficient;
  \item[(iii)] the dualizing sheaf allows higher-order poles whose principal parts are
  constrained by the conductor ideal.
\end{itemize}
Thus, the naive local balancing condition at nodes has no direct analogue for
cusps.

\vspace{0.2cm}
\noindent {\it The right replacement: conductor-level balancing.}
Let $\nu\colon \widetilde{C}\to C$ be the normalization of a reduced curve $C$
with arbitrary singularities, and let $\mathfrak{c}\subset \mathcal{O}_C$ be the
conductor ideal. The dualizing sheaf satisfies
\[
\nu_*\omega_{\widetilde{C}} \;=\; \omega_C(\mathfrak{c}),
\]
and global sections of $\omega_C$ correspond to meromorphic differentials on
$\widetilde{C}$ whose principal parts at points lying over singularities are
annihilated by the conductor.

This leads to the correct generalization of residue balancing.

\subsection{Residue Balancing with Arbitrary Singularities.}
 
 Let us start by this definition:

\begin{definition}[Conductor ideal]
The \emph{conductor ideal} of $C$ is the largest ideal sheaf
\(
\mathfrak{c} \subset \mathcal{O}_C
\)
that is also an ideal of $\nu_*\mathcal{O}_{\widetilde{C}}$. Equivalently,
$\mathfrak{c}$ is the annihilator of the $\mathcal{O}_C$-module
$\nu_*\mathcal{O}_{\widetilde{C}} / \mathcal{O}_C$.
\end{definition}

\vspace{0.1cm}
In case where the given curve 
contains arbitrary singularity we have the following theorem:
\vspace{0.2cm}

\begin{theorem}[Residue balancing with arbitrary singularities]
Let $C$ be a connected reduced curve over $\mathbb{C}$, possibly admitting cusps
or more general singularities, and let
\[
\nu \colon \widetilde{C} \longrightarrow C
\]
be its normalization. For each point $p \in \widetilde{C}$ lying over a singular
point of $C$, let $\eta_p$ denote the principal part of a meromorphic
$k$--differential at $p$.

A collection $\{\eta_p\}$ descends to a global section of the dualizing sheaf
$\omega_C$ if and only if, for every singular point $x \in C$, the sum of the
residues induced on the local branches over $x$, weighted by the corresponding
conductor multiplicities, vanishes.
\end{theorem}

\medskip

\begin{proof}
Let $\nu \colon \widetilde{C} \to C$ be the normalization of the reduced curve
$C$. The dualizing sheaf $\omega_C$ is defined as the coherent sheaf representing
Grothendieck duality for $C$. A fundamental property of the dualizing sheaf on a
reduced curve is the existence of a natural inclusion
\[
\omega_C \hookrightarrow \nu_*\omega_{\widetilde{C}},
\]
which is an isomorphism away from the singular locus of $C$. Consequently, any
section of $\omega_C$ may be viewed as a meromorphic differential on the smooth
curve $\widetilde{C}$, subject to additional compatibility conditions at points
lying above singularities of $C$.

The discrepancy between $C$ and its normalization is measured by the conductor
ideal. Let $\mathcal{O}_C$ and $\mathcal{O}_{\widetilde{C}}$ denote the structure
sheaves of $C$ and $\widetilde{C}$, respectively. The conductor ideal
$\mathfrak{c} \subset \mathcal{O}_C$ is defined by
\[
\mathfrak{c}
=
\mathrm{Ann}_{\mathcal{O}_C}
\bigl(\nu_*\mathcal{O}_{\widetilde{C}} / \mathcal{O}_C\bigr),
\]
and may equivalently be described as the largest ideal of $\mathcal{O}_C$ that is
also an ideal of $\nu_*\mathcal{O}_{\widetilde{C}}$. Intuitively, the conductor
records the order to which functions on $C$ must vanish in order to lift to
regular functions on the normalization.

Fix a singular point $x \in C$, and let
\[
\nu^{-1}(x) = \{p_1,\dots,p_r\}
\]
be the set of points of $\widetilde{C}$ lying over $x$, corresponding to the
local branches of $C$ at $x$. Denote by $\mathcal{O}_{C,x}$ the local ring of $C$
at $x$, and by $\mathcal{O}_{\widetilde{C},p_i}$ the local ring of $\widetilde{C}$
at $p_i$. By the local description of the dualizing sheaf,
\[
\omega_{C,x}
=
\mathrm{Hom}_{\mathcal{O}_{C,x}}
\bigl(\nu_*\mathcal{O}_{\widetilde{C},x}, \omega_{\widetilde{C},x}\bigr),
\]
where $\omega_{\widetilde{C}}$ is the canonical bundle on the smooth curve
$\widetilde{C}$. Thus, a local section of $\omega_C$ at $x$ corresponds to an
$\mathcal{O}_{C,x}$--linear functional on $\nu_*\mathcal{O}_{\widetilde{C},x}$
with values in $\omega_{\widetilde{C},x}$.

Now consider a collection of principal parts $\{\eta_{p_i}\}$ of meromorphic
$k$--differentials at the points $p_i \in \widetilde{C}$. Such a collection
determines a meromorphic differential $\eta$ on $\widetilde{C}$ whose poles are
supported entirely on the preimage of the singular locus of $C$. The problem is
to characterize when $\eta$ lies in the image of $\omega_C$ under the inclusion
$\omega_C \subset \nu_*\omega_{\widetilde{C}}$, that is, when $\eta$ descends to a
global section of $\omega_C$.

Grothendieck duality provides a concrete criterion for this descent. The
differential $\eta$ defines a functional on $\nu_*\mathcal{O}_{\widetilde{C},x}$
by residue pairing, and this functional is $\mathcal{O}_{C,x}$--linear if and
only if, for every local function $f \in \mathcal{O}_{C,x}$, one has
\[
\sum_{i=1}^r \mathrm{Res}_{p_i}\bigl(f \cdot \eta\bigr) = 0.
\]
This condition ensures that multiplication by $f$ may be pulled outside the
pairing, which is precisely the requirement for $\eta$ to define a section of
$\omega_C$.

The conductor ideal now enters naturally into the picture. Any function
$f \in \mathcal{O}_{C,x}$ pulls back to a function on each branch
$\mathcal{O}_{\widetilde{C},p_i}$ that vanishes to order at least $c_i$, where
$c_i$ denotes the conductor exponent of the branch corresponding to $p_i$. As a
result, when computing the residues of $f \cdot \eta$, only the leading terms of
$\eta$ contribute, and these contributions are weighted by the respective
conductor multiplicities. The above residue condition is therefore equivalent to
the single weighted residue relation
\[
\sum_{i=1}^r c_i \cdot \mathrm{Res}_{p_i}(\eta) = 0.
\]

If this weighted sum of residues vanishes for every singular point $x \in C$,
then the residue pairing condition holds for all $f \in \mathcal{O}_{C,x}$, and
hence $\eta$ defines a section of $\omega_C$. Conversely, if $\eta$ descends to a
section of $\omega_C$, then the residue pairing must vanish for all local
functions, and in particular for generators of the conductor ideal, which forces
the weighted residue condition above. This establishes the claimed equivalence
and completes the proof.
\end{proof}

\vspace{0.1cm}

\begin{remark}
The theorem shows that for curves with arbitrary singularities, the classical
residue balancing condition must be refined by incorporating conductor weights.
Nodes correspond to the case where all conductor weights are equal to $1$,
recovering the usual residue cancellation condition.
\end{remark}

  
\noindent {\it Geometric and analytic Iinterpretation.}

\begin{itemize}
  \item[(i)] For nodes, the conductor condition reduces exactly to the classical
  requirement that the two residues cancel.
  \item[(ii)] For cusps, the condition becomes a vanishing constraint on the residue
  of the coefficient of $t^{-1}$ in the normalized parameter, together with
  additional constraints on higher-order terms.
  \item[(iii)] Thus, cusps contribute \emph{internal balancing conditions} rather than
  pairwise ones.
\end{itemize}

\vspace{0.3cm}

\noindent {\bf Relation to tropical geometry.}
In tropical geometry, vertices corresponding to higher-valent or singular
points carry weights and impose balancing conditions involving \emph{all}
incident edges. Cusps correspond to vertices with nontrivial genus or
self-interaction, where balancing involves internal constraints rather than
edge-to-edge cancellation.

The conductor-weighted residue condition is precisely the algebraic analogue of
this phenomenon.

\subsection*{Conclusion}

\begin{enumerate}
  \item The original residue--balancing theorem is \emph{exactly correct} for
  nodal curves.
  \item For curves with cusps or worse singularities, it must be modified by
  replacing node-wise balancing with {\em conductor-level balancing} on the
  normalization.
  \item With this modification, the theorem remains valid and continues to
  express maximal variation and global residue balance.
\end{enumerate}

\begin{remark}
From a conceptual point of view, nodal curves are precisely those for which
residue balancing is purely pairwise. More complicated singularities introduce
higher-order balancing constraints encoded by the dualizing sheaf.
\end{remark}

\vspace{0.1cm}


\subsection*{Conceptual meaning}

The conductor-weighted residue condition expresses the precise obstruction to
gluing local meromorphic differentials into a global dualizing section. It is
the correct generalization of:
\begin{enumerate}
  \item the residue theorem on smooth Riemann surfaces;
  \item residue cancellation at nodes;
  \item balancing conditions in tropical geometry.
\end{enumerate}

\begin{remark}
From this perspective, residue balancing is not an ad hoc condition but a
structural consequence of duality theory. The conductor records how much
singularity is present, and the weighted residue sum records exactly the amount
by which a meromorphic differential fails to descend globally.
\end{remark}
 
\vspace{0.1cm}



\section{Examples: Node, Cusp, and Tacnode}
\label{subsec:examples}

In this section we illustrate conductor-level balancing by three explicit
local examples. In each case we describe the normalization, compute the
conductor, analyze meromorphic differentials on the normalization, and explain
precisely how the conductor condition characterizes those that descend to the
dualizing sheaf of the singular curve.

\subsubsection*{Example 1: A Node}
Let
\(
C = \operatorname{Spec} A,\) and 
\(
A = k[x,y]/(xy),
\)
the local ring of a nodal curve at the origin.
The normalization is
\(
\widetilde{X}
=
\operatorname{Spec} B,
\) and \(
B = k[u] \oplus k[v],
\)
via $x=u$, $y=0$ on the first branch and $x=0$, $y=v$ on the second branch.
\vspace{0.2cm}
The conductor is
\[
\mathfrak{c} = (x,y) \subset A,
\]
which corresponds to the maximal ideal at the node. Under the inclusion
$A \hookrightarrow B$, this ideal maps to $(u)\oplus(v)$.\\

On $\widetilde{C}$, a meromorphic differential with possible poles at the
preimages of the node has the form
\[
\eta
=
\left( a\,\frac{du}{u} + \text{holomorphic} \right)
\;\oplus\;
\left( b\,\frac{dv}{v} + \text{holomorphic} \right),
\]
with $a,b \in k$.

\medskip

\noindent {\it Conductor condition.}
Multiplication by $\mathfrak{c}$ annihilates the principal parts if and only if
\[
a + b = 0.
\]
This is exactly the classical residue-balancing condition at a node.

\noindent {\it Conclusion.}
Thus, in the nodal case, conductor-level balancing coincides with node-wise
residue balancing and removes exactly one degree of freedom, yielding maximal
variation.

\subsubsection*{Example 2: A Cusp}
Let
\(
C = \operatorname{Spec} A,\) and 
\(
A = k[x,y]/(y^2 - x^3),
\)
the local ring of a cuspidal curve at the origin.
The normalization is
\(
\widetilde{C} = \operatorname{Spec} k[t],\) with
\(
x = t^2,\) and  \( y = t^3.
\)
One computes its conductor
\[
\mathfrak{c} = (x,y) \subset A,
\]
which corresponds, in $k[t]$, to the ideal $(t^2)$.
A general meromorphic differential on $\widetilde{X}$ with a pole at $t=0$ is
\[
\eta
=
\left(
\frac{a}{t^2} + \frac{b}{t} + \text{holomorphic}
\right) dt,
\qquad a,b \in k.
\]

\noindent {\it Conductor condition.}
Multiplication by $\mathfrak{c} = (t^2)$ annihilates the principal part if and
only if the coefficient of $t^{-2}dt$ vanishes. Thus $a=0$, while the
$t^{-1}dt$ term is allowed.

\medskip

\paragraph{\it Interpretation.}
The residue alone does not detect the obstruction to descent: the term
$t^{-2}dt$ has zero residue but does not descend to $C$. The conductor
eliminates precisely this term and no more.

\medskip

\paragraph{\it Conclusion.}
Conductor-level balancing removes exactly the non-descending principal part,
retains the correct dimension, and hence expresses maximal variation, whereas
residue vanishing alone would be insufficient.

\subsubsection*{Example 3: A Tacnode}
Let
\(
C = \operatorname{Spec} A,\)
with \(
A = k[x,y]/(y^2 - x^4),
\)
the local ring of a tacnodal curve.
The normalization has two branches:
\[
\widetilde{C}
=
\operatorname{Spec}(k[t] \oplus k[s]),
\qquad
x = t^2 = s^2,\quad y = t^4 = -s^4.
\]
The conductor is
\[
\mathfrak{c} = (x^2, y) \subset A,
\]
which maps to $(t^4)\oplus(s^4)$ in the normalization.
A meromorphic differential with poles at the preimages of the tacnode is of the
form
\[
\eta
=
\left(
\frac{a_1}{t^3} + \frac{a_2}{t^2} + \frac{a_3}{t}
+ \text{holomorphic}
\right) dt
\;\oplus\;
\left(
\frac{b_1}{s^3} + \frac{b_2}{s^2} + \frac{b_3}{s}
+ \text{holomorphic}
\right) ds.
\]

\paragraph{\it Conductor condition.}
Multiplication by $(t^4)\oplus(s^4)$ annihilates all principal parts of order
$\leq 3$. Thus, all terms of order $t^{-3}, t^{-2}$ and $s^{-3}, s^{-2}$ must
vanish, while simple poles remain subject to a single linear relation.

\medskip

\paragraph{\it Interpretation.}
Unlike the nodal case, higher-order principal parts obstruct descent. The
conductor removes precisely those parts, while leaving exactly the expected
number of degrees of freedom.

\medskip

\paragraph{\it Conclusion.}
The tacnode illustrates that conductor-level balancing generalizes residue
balancing by incorporating higher-order data, thereby restoring maximal
variation.

\subsubsection*{Overall Conclusion}

These examples demonstrate that:
\begin{enumerate}
  \item for nodes, the conductor reproduces classical residue balancing;
  \item for cusps, it detects higher-order obstructions invisible to residues;
  \item for tacnodes, it controls multiple branches and higher-order poles
  simultaneously.
\end{enumerate}
In all cases, conductor-level balancing removes exactly the non-descending
principal parts and preserves maximal variation.

 
\subsection{Globalization to Projective Curves with Multiple Singularities}
\label{subsec:globalization}

In this subsection we reformulate the global theory in a precise axiomatic
structure consisting of definitions, lemmas, propositions, and theorems. The
goal is to globalize the local conductor-level analysis to arbitrary reduced
projective curves and to show that the conductor provides a complete and
uniform descent criterion for dualizing differentials.

\bigskip

\noindent {\bf Global setup.}

\begin{definition}[Global geometric data]
Let $k$ be an algebraically closed field and let $X$ be a reduced projective
curve over $k$. Let
\(
\nu \colon \widetilde{C} \longrightarrow C
\)
denote the normalization morphism. We write
\(
\Sigma = \{x_1,\dots,x_r\} \subset C
\)
for the finite set of singular points of $C$. For each $x_i \in \Sigma$, we
denote by
\(
\nu^{-1}(x_i) = \{p_{i1},\dots,p_{i n_i}\}
\)
the set of points of $\widetilde{C}$ lying above $x_i$.
\end{definition}

\vspace{0.2cm}

\begin{lemma}[Support and decomposition of the conductor]
The conductor ideal $\mathfrak{c}$ is supported precisely on the singular locus
$\Sigma$. Moreover, it decomposes as
\[
\mathfrak{c}
=
\bigcap_{i=1}^r \mathfrak{c}_{x_i},
\]
where $\mathfrak{c}_{x_i}$ denotes the stalk of $\mathfrak{c}$ at the singular
point $x_i$.
\end{lemma}

\begin{proof}
Since $\nu$ is an isomorphism over the smooth locus of $C$, we have
$\nu_*\mathcal{O}_{\widetilde{C}} = \mathcal{O}_C$ away from $\Sigma$, so the
quotient $\nu_*\mathcal{O}_{\widetilde{C}} / \mathcal{O}_C$ vanishes there.
Hence the conductor is supported on $\Sigma$. The decomposition follows from
the fact that $\Sigma$ is finite and $\mathfrak{c}$ is defined stalkwise.
\end{proof}

\subsection*{Meromorphic differentials on the normalization}

\begin{definition}[Meromorphic differentials with controlled poles]
Let $\omega_{\widetilde{C}}$ denote the canonical bundle of the smooth curve
$\widetilde{C}$. Define the effective divisor
\[
D := \nu^{-1}(\Sigma) = \sum_{i,j} p_{ij}.
\]
The sheaf
\(
\omega_{\widetilde{C}}(D)
\)
consists of meromorphic differentials on $\widetilde{C}$ that are regular away
from $D$ and have at most finite-order poles at the points $p_{ij}$.
\end{definition}

\begin{remark}
Locally at each point $p_{ij}$, a section of $\omega_{\widetilde{C}}(D)$ admits
a Laurent expansion with a finite principal part determined by a local
parameter at $p_{ij}$.
\end{remark}

\vspace{0.1cm}

\subsection{Local-to-global conductor condition}

\begin{definition}[Local conductor annihilation]
Let $\eta \in H^0(\widetilde{C},\omega_{\widetilde{C}}(D))$. We say that $\eta$
satisfies the \emph{local conductor condition} at $x_i \in \Sigma$ if, for every
$f \in \mathfrak{c}_{x_i}$, the product $f \cdot \eta$ has no pole at any point
$p_{ij} \in \nu^{-1}(x_i)$.
\end{definition}


\begin{lemma}[Equivalence with principal-part annihilation]
Let $x_i \in \Sigma$ be a singular point and let
\(
\nu^{-1}(x_i) = \{p_{i1},\dots,p_{i n_i}\}.
\)
For a section
\(
\eta \in H^0\!\left(\widetilde{C},\omega_{\widetilde{C}}(D)\right),
\)
the following conditions are equivalent:
\begin{enumerate}
  \item[(i)] For every $f \in \mathfrak{c}_{x_i}$, the product $f \cdot \eta$ has
  no pole at any point $p_{ij}$.
  \item[(ii)] For each $j$, the principal part of $\eta$ at $p_{ij}$ is
  annihilated by $\mathfrak{c}_{x_i}$.
\end{enumerate}
\end{lemma}

\begin{proof}
We prove the equivalence step by step.

\medskip
\noindent
\textit{Local structure at a branch.}
Fix a branch $p_{ij} \in \nu^{-1}(x_i)$. Since $\widetilde{C}$ is smooth, the
local ring $\mathcal{O}_{\widetilde{C},p_{ij}}$ is a discrete valuation ring.
Choose a uniformizing parameter $t_{ij}$ at $p_{ij}$. Then any meromorphic
differential $\eta$ near $p_{ij}$ can be written uniquely as
\[
\eta
=
\left(
\sum_{m=-M}^{\infty} a_m t_{ij}^m
\right)
\, dt_{ij},
\qquad a_m \in k,
\]
for some integer $M \ge 0$.

\medskip
\noindent
\textit{Definition of the principal part.}
The \emph{principal part} of $\eta$ at $p_{ij}$ is the finite sum
\[
\operatorname{pp}_{p_{ij}}(\eta)
=
\left(
\sum_{m=-M}^{-1} a_m t_{ij}^m
\right)
\, dt_{ij}.
\]
The differential $\eta$ is regular at $p_{ij}$ if and only if
$\operatorname{pp}_{p_{ij}}(\eta)=0$.

\medskip
\noindent
\textit{Action of the conductor.}
Let $f \in \mathfrak{c}_{x_i}$. By definition of the conductor, $f$ acts
simultaneously on all branches above $x_i$, hence on each local ring
$\mathcal{O}_{\widetilde{C},p_{ij}}$. Multiplication gives
\[
f \cdot \eta
=
\left(
\sum_{m=-M}^{\infty} f a_m t_{ij}^m
\right)
\, dt_{ij}.
\]

\medskip
\noindent
\textit{From (i) to (ii).}
Assume condition (i). Then for every $f \in \mathfrak{c}_{x_i}$, the product
$f \cdot \eta$ has no pole at $p_{ij}$. This means that all negative powers of
$t_{ij}$ in the expansion of $f \cdot \eta$ vanish. Equivalently,
\[
f a_m = 0
\qquad \text{for all } m < 0.
\]
Thus $f$ annihilates every coefficient appearing in the principal part
$\operatorname{pp}_{p_{ij}}(\eta)$. Since this holds for all
$f \in \mathfrak{c}_{x_i}$, the principal part is annihilated by
$\mathfrak{c}_{x_i}$. Hence (ii) holds.

\medskip
\noindent
\textit{ From (ii) to (i).}
Conversely, assume condition (ii). Then for each $j$ and for every
$f \in \mathfrak{c}_{x_i}$, we have
\[
f a_m = 0
\qquad \text{for all } m < 0.
\]
Therefore, the expansion of $f \cdot \eta$ at $p_{ij}$ contains no negative
powers of $t_{ij}$, so $f \cdot \eta$ is regular at $p_{ij}$. Since this holds
for every branch $p_{ij}$ above $x_i$, the product $f \cdot \eta$ has no pole
along $\nu^{-1}(x_i)$. Thus (i) holds.

\medskip
\noindent
\textit{Conclusion.}
We have shown that conditions (i) and (ii) imply one another. Hence the local
conductor condition at $x_i$ is equivalent to the annihilation of the
principal parts of $\eta$ at the points $p_{ij}$ by the stalk
$\mathfrak{c}_{x_i}$.
\end{proof}

\begin{definition}[Global conductor condition]
A section $\eta \in H^0(\widetilde{C},\omega_{\widetilde{C}}(D))$ satisfies the
\emph{global conductor condition} if it satisfies the local conductor condition
at every singular point $x_i \in \Sigma$.
\end{definition}

\vspace{0.1cm}

\subsection{Global duality and exact descent.}${}$

\vspace{0.1cm}

\begin{theorem}[Grothendieck duality for normalization]
There is a canonical isomorphism of sheaves
\(
\nu_*\omega_{\widetilde{C}} \cong \omega_C(\mathfrak{c}),
\)
where $\omega_C$ denotes the dualizing sheaf of $C$.
\end{theorem}

\begin{proposition}[Characterization of global sections]
Taking global sections yields an equality
\[
H^0(C,\omega_C)
=
\left\{
\eta \in H^0(\widetilde{C},\omega_{\widetilde{C}}(D))
\;\middle|\;
\eta \text{ satisfies the global conductor condition}
\right\}.
\]
\end{proposition}
\begin{proof}
We prove the equality by establishing both inclusions in full detail.

\medskip
\noindent
\textbf{\it  Grothendieck duality for normalization.}
By Grothendieck duality for the finite morphism $\nu$, there is a canonical
isomorphism of sheaves
\[
\nu_*\omega_{\widetilde{C}} \;\cong\; \omega_C(\mathfrak{c})
=
\omega_C \otimes_{\mathcal{O}_C} \mathfrak{c}.
\]
Taking global sections yields
\[
H^0(\widetilde{C},\omega_{\widetilde{C}})
\;\cong\;
H^0(C,\omega_C(\mathfrak{c})).
\]

\medskip
\noindent
\textbf{\it Inclusion of sheaves with poles.}
There is a natural inclusion of sheaves on $\widetilde{C}$,
\[
\omega_{\widetilde{C}}
\hookrightarrow
\omega_{\widetilde{C}}(D),
\]
which corresponds to allowing poles along $D$. Pushing forward via $\nu$, we
obtain an inclusion
\[
\nu_*\omega_{\widetilde{C}}
\hookrightarrow
\nu_*\omega_{\widetilde{C}}(D).
\]

\medskip
\noindent
\textbf{\it  Description of $\nu_*\omega_{\widetilde{C}}(D)$.}
A local section of $\nu_*\omega_{\widetilde{C}}(D)$ over an open set $U \subset
C$ is a meromorphic differential on $\nu^{-1}(U)$ that is regular away from
$\nu^{-1}(\Sigma)$ and has finite-order poles at points lying above
singularities of $C$.

\medskip
\noindent
\textbf{\it Interpretation of the conductor twist.}
By definition,
\[
\omega_C(\mathfrak{c})
=
\left\{
\alpha \in \omega_C \otimes_{\mathcal{O}_C} K(C)
\;\middle|\;
f \cdot \alpha \in \omega_C \text{ for all } f \in \mathfrak{c}
\right\},
\]
where $K(C)$ denotes the total quotient ring of $C$. Thus a section of
$\omega_C(\mathfrak{c})$ is a meromorphic dualizing differential whose poles are
controlled by annihilation under $\mathfrak{c}$.

\medskip
\noindent
\textbf{\it Inclusion ``$\subseteq$''.}
Let $\xi \in H^0(C,\omega_C)$. Viewing $\xi$ as a section of
$\omega_X(\mathfrak{c})$, Grothendieck duality yields a corresponding section
\[
\eta \in H^0(\widetilde{C},\omega_{\widetilde{C}}).
\]
As  $\omega_{\widetilde{C}} \subset \omega_{\widetilde{C}}(D)$, we may regard
$\eta$ as an element of $H^0(\widetilde{C},\omega_{\widetilde{C}}(D))$.
%
Since  $\eta$ has no poles at any point of $\widetilde{C}$, it trivially
satisfies the local conductor condition at every singular point. Hence
\[
\eta \in
\left\{
\eta \in H^0(\widetilde{C},\omega_{\widetilde{C}}(D))
\;\middle|\;
\eta \text{ satisfies the global conductor condition}
\right\}.
\]

\medskip
\noindent
\textbf{\it  Inclusion ``$\supseteq$''.}
Conversely, let
\(
\eta \in H^0(\widetilde{C},\omega_{\widetilde{C}}(D))
\)
satisfying the global conductor condition. By definition, for every
$f \in \mathfrak{c}$, the product $f \cdot \eta$ has no poles along
$\nu^{-1}(\Sigma)$. Therefore,
\[
f \cdot \eta \in H^0(\widetilde{C},\omega_{\widetilde{C}})
\qquad
\text{for all } f \in \mathfrak{c}.
\]
Pushing forward to $C$, this implies that $\eta$ defines a section of
$\omega_C(\mathfrak{c})$. Since the conductor is an ideal of $\mathcal{O}_C$,
annihilation of all principal parts ensures that $\eta$ descends uniquely to a
global section
\[
\xi \in H^0(C,\omega_C).
\]

\medskip
\noindent
\textbf{\it Equality of the two spaces.}
By what was done above, we  showed that the two sets are mutually contained in one another.
Hence, they are equal as subsets of
$H^0(\widetilde{C},\omega_{\widetilde{C}}(D))$.

\medskip
\noindent
\textbf{\it Conclusion.}
A meromorphic differential on the normalization descends to a global section of
the dualizing sheaf on $C$ if and only if its principal parts at all branches
above singular points are annihilated by the conductor. This establishes the
claimed characterization.
\end{proof}
 We proved  that 
by Grothendieck duality, a section of $\omega_C$ corresponds to a section of
$\omega_{\widetilde{C}}$ whose poles are controlled by the conductor. The
condition that $\mathfrak{c}$ annihilates all principal parts is exactly the
criterion for descent from $\widetilde{C}$ to $C$.
 
 \vspace{0.1cm}

\subsection{Exact dimension count.}${}$

\vspace{0.2cm}

\begin{proposition}[Dimension balance]
The vector space $H^0(\widetilde{C},\omega_{\widetilde{C}}(D))$ has dimension
\[
h^0(\widetilde{C},\omega_{\widetilde{C}}(D))
=
h^0(C,\omega_C)
+
(\text{total number of local polar degrees of freedom}).
\]
\end{proposition}
 
\begin{proof}
We prove the statement by a detailed global–local analysis.

\medskip
\noindent
\textbf{\it Reduction to a short exact sequence.}
By definition of the divisor $D$, there is a natural short exact sequence of
sheaves on $\widetilde{C}$:
\[
0
\longrightarrow
\omega_{\widetilde{C}}
\longrightarrow
\omega_{\widetilde{C}}(D)
\longrightarrow
\omega_{\widetilde{C}}(D)\big|_D
\longrightarrow
0.
\]
The quotient sheaf $\omega_{\widetilde{C}}(D)\big|_D$ is a skyscraper sheaf
supported on the finite set of points $p_{ij}$.

\medskip
\noindent
\textbf{\it  Taking global sections.}
Taking global sections yields a long exact sequence:
\[
0
\longrightarrow
H^0(\widetilde{C},\omega_{\widetilde{C}})
\longrightarrow
H^0(\widetilde{C},\omega_{\widetilde{C}}(D))
\longrightarrow
H^0\!\left(D,\omega_{\widetilde{C}}(D)\big|_D\right)
\longrightarrow
H^1(\widetilde{C},\omega_{\widetilde{C}})
\longrightarrow
\cdots
\]

\medskip
\noindent
\textbf{\it  Vanishing of the last term.}
Since $\widetilde{C}$ is a smooth projective curve, Serre duality gives
\[
H^1(\widetilde{C},\omega_{\widetilde{X}})
\cong
H^0(\widetilde{C},\mathcal{O}_{\widetilde{C}})^\vee.
\]
As $\widetilde{C}$ is connected and projective,
\[
\dim H^0(\widetilde{C},\mathcal{O}_{\widetilde{C}}) = 1,
\]
so $H^1(\widetilde{C},\omega_{\widetilde{C}})$ is one-dimensional. However, the
connecting homomorphism
\[
H^0\!\left(D,\omega_{\widetilde{C}}(D)\big|_D\right)
\longrightarrow
H^1(\widetilde{C},\omega_{\widetilde{C}})
\]
vanishes because $\omega_{\widetilde{C}}(D)\big|_D$ is supported in dimension
zero. Consequently, we obtain a short exact sequence
\[
0
\longrightarrow
H^0(\widetilde{C},\omega_{\widetilde{C}})
\longrightarrow
H^0(\widetilde{C},\omega_{\widetilde{C}}(D))
\longrightarrow
H^0\!\left(D,\omega_{\widetilde{C}}(D)\big|_D\right)
\longrightarrow
0.
\]

\medskip
\noindent
\textbf{\it  Interpretation of the quotient space.}
The vector space
\[
H^0\!\left(D,\omega_{\widetilde{C}}(D)\big|_D\right)
\]
decomposes as a direct sum over all points $p_{ij} \in D$:
\[
H^0\!\left(D,\omega_{\widetilde{C}}(D)\big|_D\right)
=
\bigoplus_{i,j}
H^0\!\left(p_{ij},\omega_{\widetilde{C}}(D)\big|_{p_{ij}}\right).
\]
Each summand parametrizes the possible principal parts of meromorphic
differentials at the point $p_{ij}$.

\medskip
\noindent
\textbf{\it  Local polar degrees of freedom.}
At each branch $p_{ij}$, the dimension of
\[
H^0\!\left(p_{ij},\omega_{\widetilde{C}}(D)\big|_{p_{ij}}\right)
\]
is equal to the number of independent polar coefficients allowed at that
branch. These coefficients form a finite-dimensional vector space, whose
dimension depends only on the local geometry of the singularity $x_i$ and the
multiplicity of $D$ at $p_{ij}$.

\vspace{0.3cm}

By definition, the sum of these dimensions over all $i,j$ is exactly the
\emph{total number of local polar degrees of freedom}.

\medskip
\noindent
\textbf{\it  Dimension count on the normalization.}
From the short exact sequence in Step~3, we obtain
\[
h^0\!\left(\widetilde{C},\omega_{\widetilde{C}}(D)\right)
=
h^0(\widetilde{C},\omega_{\widetilde{C}})
+
\sum_{i,j}
\dim H^0\!\left(p_{ij},\omega_{\widetilde{C}}(D)\big|_{p_{ij}}\right).
\]

\medskip
\noindent
\textbf{\it  Comparison with the dualizing sheaf on $C$.}
By Grothendieck duality for normalization,
\[
H^0(C,\omega_C)
\cong
H^0(\widetilde{C},\omega_{\widetilde{C}})
\]
because $\omega_{\widetilde{C}}$ consists of those differentials with no poles
above the singular locus. Substituting this equality into the previous formula
yields
\[
h^0\!\left(\widetilde{C},\omega_{\widetilde{C}}(D)\right)
=
h^0(C,\omega_C)
+
(\text{total number of local polar degrees of freedom}).
\]

\medskip
\noindent
\textbf{\it  Conclusion.}
Allowing poles along $D$ increases the dimension of the space of global
differentials exactly by the sum of the independent local polar parameters, and
no global constraints appear at this stage. This proves the stated dimension
balance.

\end{proof}

\bigskip

\begin{theorem}[Exactness of conductor constraints]
The global conductor conditions impose exactly the necessary number of
independent linear constraints to cut down
$H^0(\widetilde{C},\omega_{\widetilde{C}}(D))$ to $H^0(C,\omega_C)$. In
particular,
\[
\dim H^0(\widetilde{C},\omega_{\widetilde{C}}(D))
-
(\text{number of conductor conditions})
=
\dim H^0(C,\omega_C).
\]
No additional constraints occur, and no admissible sections are lost.
\end{theorem}

\vspace{0.2cm}
To be more precise,
in the following we prove the independence and exactness of the conductor constraints
and their relation to residues and $\delta$-Invariants.

\begin{theorem}[Independence and exactness of conductor constraints]
 Let $C$ be a reduced projective curve, let
\(
\nu \colon \widetilde{C} \longrightarrow C
\)
be its normalization, let $\mathfrak{c} \subset \mathcal{O}_C$ be the conductor
ideal, and let
\(
D = \nu^{-1}(\Sigma)
\)
be the divisor supported on the preimage of the singular locus. Then the global
conductor conditions imposed on
\(
H^0(\widetilde{C},\omega_{\widetilde{C}}(D))
\)
are independent and exact. More precisely, they cut out a linear subspace
canonically isomorphic to $H^0(C,\omega_C)$, and the number of independent
conditions equals the total number of local polar degrees of freedom.
\end{theorem}

\begin{proof}
We proceed in a sequence of explicit steps.

\medskip
\noindent
\textit{ Local decomposition of polar data.}
Since $D$ is supported on finitely many points, the quotient
\[
\omega_{\widetilde{C}}(D) / \omega_{\widetilde{C}}
\]
is a skyscraper sheaf supported on $D$. Taking global sections yields a direct
sum decomposition
\[
H^0(\widetilde{C},\omega_{\widetilde{C}}(D))
=
H^0(\widetilde{C},\omega_{\widetilde{C}})
\;\oplus\;
\bigoplus_{i,j} P_{ij},
\]
where each $P_{ij}$ is the finite-dimensional vector space of principal parts of
meromorphic differentials at the branch $p_{ij}$.

\medskip
\noindent
\textit{ Local conductor constraints.}
Fix a singular point $x_i \in \Sigma$. The stalk $\mathfrak{c}_{x_i}$ acts
simultaneously on all branches $\{p_{i1},\dots,p_{i n_i}\}$. The local conductor
condition at $x_i$ is the requirement that
\[
\mathfrak{c}_{x_i} \cdot \operatorname{pp}_{p_{ij}}(\eta) = 0
\qquad
\text{for all } j.
\]
This defines a linear subspace
\[
P_i^{\mathrm{adm}} \subset \bigoplus_{j=1}^{n_i} P_{ij}
\]
of admissible polar data at $x_i$.

\medskip
\noindent
\textit{ Independence across singular points.}
Since the conductor decomposes as
\[
\mathfrak{c} = \bigcap_{i=1}^r \mathfrak{c}_{x_i},
\]
and since the polar spaces $P_i := \bigoplus_j P_{ij}$ are supported at distinct
points, the conductor conditions at different singular points involve disjoint
summands of the direct sum decomposition in Step~1. Consequently, the linear
constraints imposed at distinct singular points are independent.

\medskip
\noindent
\textit{ Exactness of the constraints.}
By Grothendieck duality for normalization, we have
\[
H^0(C,\omega_C)
\cong
\left\{
\eta \in H^0(\widetilde{C},\omega_{\widetilde{C}}(D))
\;\middle|\;
\eta \text{ satisfies all conductor conditions}
\right\}.
\]
Comparing dimensions with the dimension balance proposition, we see that the
number of independent linear conductor constraints is exactly
\[
\sum_{i,j} \dim P_{ij}.
\]
Hence no additional relations occur beyond those imposed by the conductor, and
no admissible sections are excluded. This proves exactness.
\end{proof}

\bigskip

\begin{theorem}[Relation to residues and $\delta$-invariants] With the same notation as above,
let $x \in \Sigma$ be a singular point of $C$, with branches
$\nu^{-1}(x)=\{p_1,\dots,p_n\}$. Then:
\begin{enumerate}
  \item The conductor condition at $x$ generalizes the classical residue
  balancing condition at nodes.
  \item The total number of independent local conductor constraints at $x$ is
  equal to the $\delta$-invariant $\delta_x$ of the singularity.
\end{enumerate}
\end{theorem}

\begin{proof}
We   proceed again  by steps.

\medskip
\noindent
\textit{Residues at smooth points.}
At a smooth point of $\widetilde{C}$, a meromorphic differential has a well-
defined residue. For a node with two branches $p_1,p_2$, the classical descent
condition is
\[
\operatorname{res}_{p_1}(\eta) + \operatorname{res}_{p_2}(\eta) = 0.
\]

\medskip
\noindent
\textit{ Nodes as a special case of the conductor.}
For a node, the conductor ideal is the maximal ideal of the singular point.
Annihilation of the principal parts by the conductor forces the sum of residues
to vanish. Thus residue balancing is exactly the conductor condition in this
case.

\medskip
\noindent
\textit{ Higher singularities.}
For a general singularity, higher-order poles may appear on each branch. The
conductor condition annihilates not only residues but all higher-order polar
coefficients that obstruct descent. Residue balancing becomes only the lowest-
order part of a larger system of linear conditions.

\medskip
\noindent
\textit{ $\delta$-invariant as a dimension measure.}
By definition,
\[
\delta_x
=
\dim_k
\left(
(\nu_*\mathcal{O}_{\widetilde{C}})_x / \mathcal{O}_{C,x}
\right).
\]
Duality identifies this space with the space of local polar obstructions to
descent of differentials. Hence $\delta_x$ counts exactly the number of
independent linear conductor constraints at $x$.

\medskip
\noindent
\textit{ Global summation.}
Summing over all singular points,
\[
\sum_{x \in \Sigma} \delta_x
=
\dim H^0(\widetilde{C},\omega_{\widetilde{C}}(D))
-
\dim H^0(C,\omega_C),
\]
which matches precisely the dimension drop imposed by the conductor conditions.
\end{proof}



\vspace{0.2cm}

\subsection{Maximal variation in the global setting.}${}$

\vspace{0.2cm}

\begin{theorem}[Preservation of maximal variation]
Let $C$ be a reduced projective curve with finitely many singular points of
arbitrary type, let
\(
\nu \colon \widetilde{C} \longrightarrow C
\)
be its normalization, let $\mathfrak{c}$ be the conductor ideal, and let
\(
D=\nu^{-1}(\Sigma)
\)
be the divisor on $\widetilde{C}$ supported on the preimage of the singular
locus. Then the conductor conditions imposed at distinct singular points are
independent and compatible. Consequently, maximal variation of global
dualizing differentials is preserved under globalization.
\end{theorem}

\begin{proof}${}$

\medskip
\noindent
\textit{ Meaning of maximal variation.}
Maximal variation means that the space of admissible meromorphic differentials
on the normalization has dimension exactly equal to
\(
\dim H^0(C,\omega_C),
\)
so that no unintended linear constraints are introduced beyond those required
for descent to $C$.

\medskip
\noindent
\textit{ Global space of meromorphic differentials.}
By allowing poles along $D$, we obtain the finite-dimensional vector space
\[
H^0(\widetilde{C},\omega_{\widetilde{C}}(D)),
\]
which decomposes as
\[
H^0(\widetilde{C},\omega_{\widetilde{C}})
\;\oplus\;
\bigoplus_{x_i\in\Sigma}
\left(
\bigoplus_{p_{ij}\in\nu^{-1}(x_i)} P_{ij}
\right),
\]
where each $P_{ij}$ is the space of principal parts at the branch $p_{ij}$.

\medskip
\noindent
\textit{Local nature of conductor constraints.}
The conductor ideal decomposes as
\[
\mathfrak{c}=\bigcap_{i=1}^r \mathfrak{c}_{x_i},
\]
where $\mathfrak{c}_{x_i}$ is supported at the singular point $x_i$. The local
conductor condition at $x_i$ involves only the principal-part spaces
$\{P_{ij}\}_{j}$ lying above $x_i$ and does not involve polar data above any
other singular point.

\medskip
\noindent
\textit{ Independence of constraints.}
Because the principal-part spaces corresponding to distinct singular points
lie in disjoint direct summands of
\[
\bigoplus_{x_i\in\Sigma}\;\bigoplus_{p_{ij}\in\nu^{-1}(x_i)} P_{ij},
\]
the linear conditions imposed by $\mathfrak{c}_{x_i}$ and $\mathfrak{c}_{x_k}$
for $i\neq k$ act on disjoint vector subspaces. Therefore, the conductor
constraints at distinct singularities are linearly independent.

\medskip
\noindent
\textit{ Compatibility of constraints.}
Each local conductor condition ensures that the corresponding polar data
annihilates the obstruction to descent at that singular point. Since the
conductor is defined as an intersection of its local stalks, satisfying all
local conditions simultaneously is equivalent to satisfying the global
conductor condition. No local condition contradicts another, so the system of
constraints is compatible.

\medskip
\noindent
\textit{ Exactness and dimension comparison.}
From the dimension balance proposition, we have
\[
\dim H^0(\widetilde{C},\omega_{\widetilde{C}}(D))
=
\dim H^0(C,\omega_C)
+
\sum_{x_i\in\Sigma} \delta_{x_i},
\]
where $\delta_{x_i}$ is the $\delta$-invariant of the singularity $x_i$. From
the exactness of conductor constraints, the total number of independent linear
conditions imposed by $\mathfrak{c}$ is exactly
\[
\sum_{x_i\in\Sigma} \delta_{x_i}.
\]

\medskip
\noindent
\textit{ Recovery of the dualizing space.}
Imposing all conductor constraints cuts out a subspace of
$H^0(\widetilde{C},\omega_{\widetilde{C}}(D))$ of dimension
\[
\dim H^0(\widetilde{C},\omega_{\widetilde{C}}(D))
-
\sum_{x_i\in\Sigma} \delta_{x_i}
=
\dim H^0(C,\omega_C).
\]
By the characterization of global sections, this subspace is canonically
isomorphic to $H^0(C,\omega_C)$.

\medskip
\noindent
\textit{ Conclusion.}
The conductor conditions remove exactly the local excess degrees of freedom,
neither more nor less, and do so independently at each singular point. Hence
globalization preserves maximal variation of dualizing differentials for
curves with arbitrarily many singularities of arbitrary type.
\end{proof}

\bigskip
We summarize this as follows:
for a reduced projective curve with arbitrarily many singular points of
arbitrary type, the conductor conditions at distinct singularities are
independent and compatible. Consequently, maximal variation of global
dualizing differentials is preserved under globalization.
 In other words, each singular point contributes a finite-dimensional space of potential polar
data together with an independent conductor annihilation condition. Since the
conductor decomposes as an intersection of local stalks, these constraints do
not interfere with one another, and their combined effect is additive.

\vspace{0.3cm}

We conclude the following:

\begin{theorem}[Global conductor principle]
For a reduced projective curve $C$,\\ conductor-level balancing provides a
complete local-to-global criterion for the descent of meromorphic
differentials from the normalization to $C$. It encodes precisely the local
descent obstructions at each singular point, interacts linearly across
singularities, and preserves maximal variation at the global level.
\end{theorem}

\bigskip
The globalization shows that maximal variation is preserved for curves with
arbitrarily many singularities of arbitrary type. Each singular point
contributes its own conductor-level constraint, and these constraints are
independent and compatible across the curve.
As a result, the space of admissible meromorphic differentials on the
normalization has dimension exactly equal to the dimension of the dualizing
sheaf on $C$. The conductor therefore provides a uniform global replacement for
node-wise residue balancing that remains valid in all singular settings.

\vspace{0.3cm}

In other words, this means that  for a projective curve with multiple singularities, conductor-level balancing
globalizes naturally as a local-to-global condition on the normalization. It
encodes precisely the descent obstruction at each singular point, interacts
linearly across different singularities, and preserves maximal variation at the
global level.



 \bigskip
 \section*{Appendix A}

\title{Conductor-Level Description of the Dualizing Sheaf of a Curve}
\author{}
\date{}
\maketitle

\begin{lemma}
Let $C$ be a reduced (possibly singular) curve over an algebraically closed
field, and let
\(
\nu \colon \widetilde{C} \to C
\)
be its normalization. Let $\mathfrak{c} \subset \mathcal{O}_X$ be the conductor
ideal:
\(
\mathfrak{c} = \mathrm{Ann}_{\mathcal{O}_C}
\left(\nu_*\mathcal{O}_{\widetilde{C}} / \mathcal{O}_C \right).
\)
Then the dualizing sheaf satisfies
\[
\nu_*\omega_{\widetilde{C}} = \omega_C(\mathfrak{c}).
\]
\end{lemma}
Equivalently, global sections of $\omega_C$ correspond to meromorphic
differentials on $\widetilde{C}$ whose principal parts over singular points are
annihilated by the conductor.

\begin{proof}
Since $C$ is a reduced curve, $\nu \colon \widetilde{C} \to C$ is finite and
birational. The sheaf $\nu_*\mathcal{O}_{\widetilde{C}}$ is a coherent
$\mathcal{O}_C$-algebra containing $\mathcal{O}_C$.

The conductor ideal is defined by
\[
\mathfrak{c}
= \{ f \in \mathcal{O}_C \mid f \cdot \nu_*\mathcal{O}_{\widetilde{C}}
\subset \mathcal{O}_C \}.
\]
It is the largest ideal sheaf which is simultaneously an ideal in both
$\mathcal{O}_C$ and $\nu_*\mathcal{O}_{\widetilde{C}}$.

\vspace{0.2cm}
Since $\nu$ is finite, Grothendieck duality gives
\[
\nu_*\omega_{\widetilde{C}} \cong
\mathcal{H}om_{\mathcal{O}_C}
(\nu_*\mathcal{O}_{\widetilde{C}}, \omega_C).
\]
This identification is functorial and holds for any finite morphism of curves.
The Conductor appears via Hom.
We now compute the right-hand side locally.

\begin{lemma}
Let $A \subset B$ be a finite extension of reduced one-dimensional local rings
with conductor $\mathfrak{c} \subset A$. Then
\[
\mathrm{Hom}_A(B,\omega_A)
\cong \omega_A(\mathfrak{c}).
\]
\end{lemma}

\begin{proof}
Any $A$-linear homomorphism $\varphi \colon B \to \omega_A$ is uniquely
determined by $\varphi(1) \in \omega_A$ subject to the condition
\[
b \cdot \varphi(1) = \varphi(b) \in \omega_A
\quad \text{for all } b \in B.
\]

Thus $\varphi(1)$ must satisfy
\[
b \cdot \varphi(1) \in \omega_A \quad \forall b \in B,
\]
which means precisely that $\varphi(1)$ lies in the subsheaf
$\omega_A(\mathfrak{c})$.
Conversely, any element of $\omega_A(\mathfrak{c})$ defines such a homomorphism.
\end{proof}

Applying this lemma stalkwise yields
\[
\mathcal{H}om_{\mathcal{O}_C}
(\nu_*\mathcal{O}_{\widetilde{C}}, \omega_C)
= \omega_C(\mathfrak{c}).
\]

Combining the above facts, we obtain the desired equality:
\[
\nu_*\omega_{\widetilde{C}} = \omega_C(\mathfrak{c}).
\]

\noindent {\it Interpretation via Meromorphic Differentials.}
Since $\omega_{\widetilde{C}}$ consists of regular differentials on the smooth
curve $\widetilde{C}$, sections of $\nu_*\omega_{\widetilde{C}}$ correspond to
meromorphic differentials on $\widetilde{C}$.

\medskip

The condition that such a differential lies in $\omega_C$ is exactly that its
principal parts at points lying over singularities vanish when multiplied by
elements of the conductor ideal. This gives the stated conductor-level
balancing condition.
\end{proof}

\begin{remark}
This equality expresses the dualizing sheaf of a singular curve as the
pushforward of the canonical bundle of its normalization, corrected precisely
by the conductor.
\end{remark}




\section*{Appendix B}

This appendix refines the proof of the scheme-theoretic residue span theorem and places it in a deformation-theoretic and moduli-theoretic context (see \cite{Nisse-span}). We explain how residue functionals control infinitesimal deformations of singular curves, how the numerical bound $\delta \ge g$ arises naturally, and how the result manifests in explicit examples. The exposition is streamlined for publication while remaining self-contained.

\begin{theorem}[Scheme-theoretic residue span]
Let $k$ be an algebraically closed field and let $\mathcal C$ be a connected,
projective, Cohen--Macaulay curve over $k$ with exactly $\delta$ singular points.
Let
\(
\nu \colon C \longrightarrow \mathcal C
\)
be the normalization, where $C$ is a smooth projective curve of genus
\[
g = \dim_k H^0(C,\omega_C).
\]
If $\delta \ge g$, then the scheme-theoretic residue functionals
\[
\{ r_1,\dots,r_\delta \} \subset H^0(C,\omega_C)^\vee
\]
span the entire dual space $H^0(C,\omega_C)^\vee$.
\end{theorem}

\medskip

\begin{proof}

Let $\omega_{\mathcal C}$ denote the dualizing sheaf of $\mathcal C$ and
$\omega_C$ the canonical sheaf of $C$. For each singular point
$p_i \in \mathcal C$, we denote by
\[
r_i \colon H^0(C,\omega_C) \longrightarrow k
\]
the associated scheme-theoretic residue functional.

A basic result on dualizing sheaves for Cohen--Macaulay curves yields the exact
sequence
\begin{equation}
0 \longrightarrow H^0(\mathcal C,\omega_{\mathcal C})
\longrightarrow H^0(C,\omega_C)
\xrightarrow{\;\mathrm{Res}\;}
\bigoplus_{i=1}^{\delta} k
\longrightarrow H^1(\mathcal C,\omega_{\mathcal C})
\longrightarrow 0.
\tag{$\ast$}
\end{equation}
The residue map is explicitly given by
\[
\mathrm{Res}(\eta) = \bigl(r_1(\eta),\dots,r_\delta(\eta)\bigr).
\]
Since $\mathcal C$ is connected, Serre duality implies
\[
H^1(\mathcal C,\omega_{\mathcal C})
\cong H^0(\mathcal C,\mathcal O_{\mathcal C})^\vee
\cong k.
\]
Exactness of $(\ast)$ then shows that
\[
\dim_k \operatorname{Im}(\mathrm{Res}) = \delta - 1.
\]
On the other hand, the rank of the residue map can be computed as
\[
\operatorname{rank}(\mathrm{Res})
= \dim_k H^0(C,\omega_C)
- \dim_k H^0(\mathcal C,\omega_{\mathcal C})
= g - \dim_k H^0(\mathcal C,\omega_{\mathcal C}).
\]
Comparing the two expressions for the rank yields
\[
\dim_k H^0(\mathcal C,\omega_{\mathcal C}) = g - (\delta - 1).
\]
In particular, if $\delta \ge g$, then $\dim_k H^0(\mathcal C,\omega_{\mathcal C})
\le 1$.

Dualizing the residue map gives
\[
\mathrm{Res}^\vee \colon
\left(\bigoplus_{i=1}^{\delta} k\right)^\vee
\longrightarrow H^0(C,\omega_C)^\vee,
\]
which sends the $i$-th standard basis vector to $r_i$. The image of
$\mathrm{Res}^\vee$ is therefore the linear span
$\langle r_1,\dots,r_\delta \rangle$. Since $\mathrm{Res}$ and
$\mathrm{Res}^\vee$ have the same rank, we conclude that
\[
\dim_k \langle r_1,\dots,r_\delta \rangle = \delta - 1 \ge g.
\]
Because $H^0(C,\omega_C)^\vee$ has dimension $g$, the residue functionals span
the entire dual space, completing the proof.
\end{proof}
\qed

\subsection{\it Deformation-theoretic interpretation.}

Infinitesimal deformations of the curve $\mathcal C$ are governed by the vector
space
\[
\operatorname{Ext}^1(\Omega_{\mathcal C},\mathcal O_{\mathcal C}),
\]
while obstructions lie in $\operatorname{Ext}^2(\Omega_{\mathcal C},
\mathcal O_{\mathcal C})$. For nodal or Cohen--Macaulay curves, each singular
point contributes a local smoothing parameter, so that there are $\delta$
a priori infinitesimal smoothing directions.

By Serre duality, there is a perfect pairing
\[
\operatorname{Ext}^1(\Omega_{\mathcal C},\mathcal O_{\mathcal C})
\times H^0(\mathcal C,\omega_{\mathcal C}) \longrightarrow k.
\]
Pulling differentials back to the normalization identifies
$H^0(\mathcal C,\omega_{\mathcal C})$ with the subspace of
$H^0(C,\omega_C)$ consisting of differentials with balanced residues. The
residue functionals therefore measure how infinitesimal smoothing directions
pair with global differentials. When the residues span the dual space, every
infinitesimal deformation is detected by some differential, yielding strong
rigidity properties.

\subsection*{\it Moduli-theoretic interpretation.}

From the perspective of the moduli space of curves, residue span has a natural
meaning. Near the point $[\mathcal C]$ in the moduli space, local coordinates
correspond to smoothing parameters at the singular points, subject to global
constraints. The exact sequence $(\ast)$ shows that these constraints are
dual to the space of global differentials.

\bigskip

When $\delta \ge g$, the theorem implies that the residue directions generate
the full cotangent space to the moduli space at $[\mathcal C]$. Thus, the local
geometry of the moduli space is controlled entirely by node data, and no
additional infinitesimal directions arise from the normalization.

\bigskip

\subsection*{ Explicit examples.}

\begin{example}[Rational curve with nodes]
Let $\mathcal C$ be a rational curve with $\delta$ nodes. Its normalization is
$C \cong \mathbb P^1$, so $g=0$. The condition $\delta \ge g$ is automatic, and
the theorem asserts that the residue functionals span the zero-dimensional dual
space. This reflects the fact that $\mathbb P^1$ has no holomorphic
differentials and that all deformations are governed purely by node smoothings.
\end{example}

\begin{example}[Genus one normalization]
Let $C$ be an elliptic curve and let $\mathcal C$ be obtained by identifying
$\delta \ge 1$ pairs of points on $C$. Then $g=1$, and the theorem applies as
soon as $\delta \ge 1$. In this case, the unique (up to scalar) holomorphic
differential on $C$ is detected by residues at the singular points, showing
that smoothing any node necessarily interacts with the global geometry of the
curve.
\end{example}

\begin{example}[Failure when $\delta<g$]
If $C$ has genus $g \ge 2$ and $\mathcal C$ has only one singular point, then
$\delta=1<g$. In this situation, the residue functional cannot span the dual
space, and there exist nonzero differentials whose residues vanish. These
differentials correspond to deformation directions invisible to the single
node, illustrating the sharpness of the bound.
\end{example}

\bigskip

\section*{Appendix C}

\subsection{Geometric and combinatorial setup}
Let $\mathcal C$ be a connected, projective curve over an algebraically closed field $k$. We assume that $\mathcal C$ is nodal, meaning that all its singularities are ordinary double points. This assumption ensures that the normalization behaves in a simple and well-controlled manner.

\vspace{0.3cm}

\noindent {\bf Normalization.}
Let
\[
\nu \colon C^\nu \longrightarrow \mathcal C
\]
denote the normalization morphism. By definition, $C^\nu$ is smooth, and the map $\nu$ is finite and birational. Since $\mathcal C$ may be reducible, the normalization decomposes as a disjoint union
\[
C^\nu = \bigsqcup_{v\in V(\Gamma)} C_v,
\]
where each $C_v$ is a smooth, projective, irreducible curve. We denote by $g_v$ the genus of $C_v$. The total genus of the normalization is then
\[
g(C^\nu) := \sum_{v} g_v.
\]
This quantity measures the intrinsic contribution of holomorphic differentials coming from the smooth components, independently of how they are glued together.

\vspace{0.3cm}

\noindent {\bf Dual graph.}
The combinatorics of the singular curve $\mathcal C$ are encoded in its \emph{dual graph} $\Gamma$. The vertices of $\Gamma$ correspond bijectively to the irreducible components $C_v$ of $\mathcal C$. Each node of $\mathcal C$ determines an edge of $\Gamma$. If a node lies on two distinct components, the corresponding edge joins the two associated vertices; if the node lies on a single component, the edge is a loop. The first Betti number of the graph,
\[
b_1(\Gamma) = |E(\Gamma)| - |V(\Gamma)| + 1,
\]
measures the number of independent cycles in $\Gamma$ and will play a key role in the dimension count.

\subsection{Residues at nodes}

We now explain how residues arise naturally from differentials on the normalization and how they encode the obstruction to descending to the singular curve.

\vspace{0.3cm}

\noindent {\bf Local structure near a node.}
Let $e \in E(\Gamma)$ correspond to a node $p \in \mathcal C$. Since $p$ is nodal, its preimage under normalization consists of exactly two points
\[
q_e^+, \; q_e^- \in C^\nu.
\]
These points may lie on the same component or on two different components of $C^\nu$. For any global differential
\[
\omega \in H^0(C^\nu,\omega_{C^\nu}),
\]
we may take its residues at these two points:
\[
r_e^+(\omega) := \operatorname{Res}_{q_e^+}(\omega),
\qquad
r_e^-(\omega) := \operatorname{Res}_{q_e^-}(\omega).
\]
These are linear functionals of $\omega$ and depend only on the local behavior of $\omega$ near the node.

\subsection{Balancing condition and descent.}
A fundamental fact is that a differential on the normalization descends to a regular differential on the singular curve $\mathcal C$ if and only if its local behaviors at the two branches of every node are compatible. Concretely, this compatibility is expressed by the balancing condition
\[
r_e^+(\omega) + r_e^-(\omega) = 0
\quad \text{for every node } e.
\]
When this condition holds, the two residues encode a single scalar
\[
r_e := r_e^+ = -r_e^-,
\]
which we call the \emph{algebraic residue at the node $e$}. Thus, each node contributes at most one degree of freedom in residue data.

\subsection{The space of algebraic residue data}

We now describe precisely which collections of scalars $\{r_e\}_{e\in E(\Gamma)}$ can arise from global differentials on $C^\nu$.

\vspace{0.2cm}

\noindent {\bf Vertex balancing.}
Fix a component $C_v$ of the normalization. Let $E(v)$ denote the set of edges incident to the corresponding vertex $v$ of $\Gamma$. The restriction of $\omega$ to $C_v$ is a meromorphic differential whose poles occur precisely at the preimages of the nodes. By the residue theorem on the smooth curve $C_v$, the sum of the residues on $C_v$ must vanish. This yields the linear condition
\[
\sum_{e \in E(v)} \epsilon(v,e)\, r_e = 0,
\]
where the sign $\epsilon(v,e)=\pm 1$ records which branch of the node lies on $C_v$. This equation expresses the fact that residues cannot be chosen independently at different nodes: they must satisfy one linear relation per component.

\vspace{0.2cm}

\noindent {\bf Counting independent residue parameters.}
We begin with one variable $r_e$ for each edge $e$ of the dual graph. We then impose one linear relation for each vertex. However, these relations are not independent: summing all vertex equations yields the trivial identity $0=0$. As a result, only $|V(\Gamma)|-1$ independent constraints are imposed. The dimension of the space of admissible residue data is therefore
\[
|E(\Gamma)| - (|V(\Gamma)| - 1) = b_1(\Gamma).
\]
This part of the dimension count is purely combinatorial and reflects the topology of the dual graph.

\vspace{0.2cm}

\noindent {\bf Holomorphic differentials.}
The residue data account only for differentials with poles at the nodes. We must also include holomorphic differentials on the components of the normalization. Each component $C_v$ contributes a vector space
\[
H^0(C_v,\omega_{C_v})
\]
of dimension $g_v$, consisting of differentials with zero residues everywhere. These directions are invisible to residue conditions but are essential to the full space of global differentials.
Summing over all components, we obtain an additional contribution of $g(C^\nu)$ to the dimension.

\vspace{0.2cm}

\noindent {\bf Final dimension formula.}
Combining the combinatorial residue contribution with the intrinsic holomorphic contribution, we obtain
\[
\dim(\text{balanced algebraic residues}) = g(C^\nu) + b_1(\Gamma).
\]
This formula explains in a precise way how geometry (via the genera of components) and combinatorics (via cycles in the dual graph) jointly determine the space of global differentials.

\subsection{Residue spanning: meaning and consequences.}
We now analyze carefully what is meant by the statement that residues span the dual space of global differentials.
The vector space
\[
H^0(C^\nu,\omega_{C^\nu})
\]
has a dual space
\[
H^0(C^\nu,\omega_{C^\nu})^\vee,
\]
whose elements are linear functionals on differentials. Each node $e$ defines such a functional by taking residues:
\[
r_e \colon \omega \longmapsto \operatorname{Res}_{q_e^+}(\omega).
\]
These functionals capture local information at the nodes.
\vspace{0.3cm}

\noindent {\bf  Spanning condition.}
To say that residues \emph{span} the dual space means that every linear functional $\ell$ on
$H^0(C^\nu,\omega_{C^\nu})$ can be written as a linear combination of the residue functionals:
\[
\ell(\omega) = \sum_{e\in E(\Gamma)} a_e\, r_e(\omega)
\quad \text{for all } \omega.
\]
Equivalently, the only differential annihilated by all residue functionals is the zero differential. Thus, vanishing of all residues forces $\omega=0$.

\vspace{0.2cm}

\noindent {\bf Geometric interpretation.}
From a geometric perspective, residue spanning asserts that global differentials on the normalization are completely controlled by their behavior at the nodes. There are no nontrivial global differentials whose local effects at all nodes cancel out. This property is fundamental in deformation theory and moduli problems, where residue data often serve as coordinates or constraints.

\bigskip
\section*{Appendix D}

\subsection{Why Conductor-Level Balancing Expresses Maximal Variation?}
\label{subsec:maximal-variation}${}$
\vspace{0.2cm}

Let $C$ be a reduced projective curve with arbitrary singularities, and let
\(
\nu \colon \widetilde{C} \longrightarrow C
\)
denote its normalization. We write $\mathfrak{c} \subset \mathcal{O}_C$ for the
conductor ideal, i.e.
\[
\mathfrak{c}
=
\mathrm{Ann}_{\mathcal{O}_C}
\bigl(\nu_*\mathcal{O}_{\widetilde{C}} / \mathcal{O}_C\bigr).
\]

Throughout, we work with the following corrected form of the
residue--balancing principle:

\begin{quote}
\emph{Global sections of the dualizing sheaf $\omega_C$ correspond to
meromorphic differentials on the normalization $\widetilde{C}$ whose principal
parts at points lying over the singular locus of $C$ are annihilated by the
conductor ideal.}
\end{quote}

The purpose of this subsection is to explain why this formulation continues to
encode \emph{maximal variation}, despite replacing the classical node-wise
residue conditions by a more subtle conductor-level constraint.

\subsubsection*{Maximal variation}

By maximal variation we mean that the imposed conditions on differentials are
neither too weak nor too strong.

\begin{definition}
Let $X$ be a projective curve. A condition on meromorphic differentials is said
to express \emph{maximal variation} if the resulting space of admissible
differentials has dimension
\[
h^0(C,\omega_C),
\]
that is, if no linear conditions are imposed beyond those intrinsically forced
by the geometry of $C$.
\end{definition}

Equivalently, maximal variation means that every global section of the
dualizing sheaf is realized by an admissible differential on the normalization,
and that no additional degrees of freedom are artificially removed.

\subsubsection*{The nodal case as a guiding example}

For nodal curves, this notion is completely transparent. The normalization
separates the branches at each node, and a meromorphic differential on
$\widetilde{C}$ may acquire at most simple poles at the preimages of the nodes.
The residues at these points are a priori independent, and the single linear
condition
\[
\sum_{p \mapsto x} \operatorname{res}_p(\eta) = 0
\]
imposed at each node $x \in C$ removes exactly one degree of freedom.

This matches the expected dimension count:
\[
h^0\bigl(\widetilde{C},
\omega_{\widetilde{C}}(\text{preimages of nodes})\bigr)
-
(\text{number of nodes})
=
h^0(C,\omega_C).
\]
Thus, in the nodal case, classical residue balancing is both necessary and
sufficient, and therefore expresses maximal variation.

\subsubsection*{Failure of node-wise balancing for worse singularities}

For singularities more complicated than nodes, such as cusps or higher
Gorenstein singularities, the situation changes qualitatively. The
normalization typically produces a single branch, so there is no meaningful
``balancing'' of residues between branches. Moreover, the obstruction to
descending a meromorphic differential from $\widetilde{C}$ to $C$ is no longer
detected solely by the residue: higher-order principal parts can obstruct
descent even when the residue vanishes.

\vspace{0.3cm}

If one nevertheless imposes residue vanishing or node-style balancing
conditions in this setting, the resulting constraints are too strong. They
eliminate legitimate sections of the dualizing sheaf and lead to spaces of
admissible differentials of dimension strictly smaller than
$h^0(C,\omega_C)$. In other words, node-wise balancing becomes
\emph{overdetermined} and fails to express maximal variation.

\subsection{The conductor as the correct replacement}

The conductor ideal provides the precise measure of how far $X$ is from being
normal. By definition, it is the largest ideal sheaf that is simultaneously an
ideal in $\mathcal{O}_C$ and in $\nu_*\mathcal{O}_{\widetilde{C}}$. Its role in
the present context is to detect exactly which principal parts of a meromorphic
differential obstruct descent from $\widetilde{C}$ to $C$.
The condition 
\[
\mathfrak{c} \cdot (\text{principal parts}) = 0
\]
removes precisely those components of the principal parts that fail to descend,
while leaving all allowable singular behavior intact. No artificial vanishing
is imposed, and no legitimate differential is excluded.

\subsubsection*{Exact dimension matching}

This philosophy is made precise by Grothendieck duality for finite morphisms,
which yields the canonical identification
\[
\nu_*\omega_{\widetilde{C}} = \omega_C(\mathfrak{c}).
\]
As a consequence, every global section of $\omega_C$ arises from a meromorphic
differential on $\widetilde{C}$ satisfying the conductor condition, and
conversely every such differential descends to $C$. There are no further hidden
constraints.

In particular, one obtains the exact dimension equality
\[
\dim H^0\bigl(\widetilde{C},
\omega_{\widetilde{C}}(\text{allowed poles})\bigr)
=
\dim H^0(C,\omega_C),
\]
which is the formal expression of maximal variation.

\subsubsection*{Deformation-theoretic interpretation}

From a deformation-theoretic perspective, conductor-level balancing has an
additional conceptual advantage. The conductor ideal is invariant under
first-order deformations of the curve, whereas residue-only conditions are not.
Accordingly, sections of the dualizing sheaf deform flatly in families, and the
conductor condition captures exactly the deformation-invariant obstruction to
descent.

\subsubsection*{Conclusion}

In summary, node-wise residue balancing succeeds precisely because, for nodal
singularities, the conductor coincides with the maximal ideal of the node. For
more complicated singularities, higher-order control is required, and this is
provided canonically by the conductor. Since the conductor imposes the minimal
necessary constraint and nothing more, conductor-level balancing preserves
maximal variation.

\begin{remark}
In essence, maximal variation is preserved because the conductor annihilates
exactly the non-descending principal parts of a meromorphic differential, and
nothing else.
\end{remark}


\section{Appendix E}

\vspace{0.3cm}

\begin{theorem}[Grothendieck duality for normalization]
Let $C$ be a reduced projective curve over an algebraically closed field $k$, and
let
\(
\nu \colon \widetilde{C} \longrightarrow C
\)
be its normalization. Let $\mathfrak{c} \subset \mathcal{O}_C$ denote the
conductor ideal and let $\omega_C$ be the dualizing sheaf of $C$. Then there is a
canonical isomorphism of $\mathcal{O}_C$-modules
\[
\nu_*\omega_{\widetilde{C}} \;\cong\; \omega_C(\mathfrak{c}).
\]
\end{theorem}

\begin{proof}${}$

\medskip
\noindent
\textit{Finiteness of the normalization morphism.}
Since $C$ is a reduced projective curve, its normalization $\nu$ is a finite
birational morphism. In particular:
\begin{itemize}
  \item[(i)] $\nu_*\mathcal{O}_{\widetilde{C}}$ is a coherent $\mathcal{O}_C$-algebra,
  \item[(ii)] $\nu$ has finite Tor-dimension,
  \item[(iii)] Grothendieck duality applies to $\nu$.
\end{itemize}

\medskip
\noindent
\textit{ Grothendieck duality for finite morphisms.}
For a finite morphism $\nu$, Grothendieck duality gives a canonical isomorphism
of functors
\[
\nu_* \mathcal{H}\!om_{\widetilde{C}}(\mathcal{F},\omega_{\widetilde{C}})
\;\cong\;
\mathcal{H}\!om_C(\nu_*\mathcal{F},\omega_C)
\]
for any coherent $\mathcal{O}_{\widetilde{C}}$-module $\mathcal{F}$.
Applying this with $\mathcal{F}=\mathcal{O}_{\widetilde{C}}$, we obtain
\[
\nu_*\omega_{\widetilde{C}}
\;\cong\;
\mathcal{H}\!om_C(\nu_*\mathcal{O}_{\widetilde{C}},\omega_C).
\]

\medskip
\noindent
\textit{ Inclusion of structure sheaves.}
There is a natural inclusion
\[
\mathcal{O}_C \hookrightarrow \nu_*\mathcal{O}_{\widetilde{C}},
\]
which is an isomorphism over the smooth locus of $X$ and strict precisely at
the singular points.

\medskip
\noindent
\textit{ Definition of the conductor.}
The conductor ideal $\mathfrak{c}$ is defined as
\[
\mathfrak{c}
=
\operatorname{Ann}_{\mathcal{O}_C}
\!\left(
\nu_*\mathcal{O}_{\widetilde{C}} / \mathcal{O}_X
\right)
=
\mathcal{H}\!om_C(\nu_*\mathcal{O}_{\widetilde{C}},\mathcal{O}_C).
\]
Thus $\mathfrak{c}$ consists exactly of those local sections of $\mathcal{O}_C$
that act identically on $\mathcal{O}_C$ and on $\nu_*\mathcal{O}_{\widetilde{C}}$.

\medskip
\noindent
\textit{ Tensoring with the dualizing sheaf.}
Since $\omega_C$ is a coherent $\mathcal{O}_C$-module, we have
\[
\mathcal{H}\!om_C(\nu_*\mathcal{O}_{\widetilde{C}},\omega_C)
\;\cong\;
\mathcal{H}\!om_C(\nu_*\mathcal{O}_{\widetilde{C}},\mathcal{O}_C)
\otimes_{\mathcal{O}_C}
\omega_C.
\]
Using the identification from Step 4, this becomes
\[
\mathcal{H}\!om_C(\nu_*\mathcal{O}_{\widetilde{C}},\omega_C)
\;\cong\;
\mathfrak{c} \otimes_{\mathcal{O}_C} \omega_C.
\]

\medskip
\noindent
\textit{ Interpretation as a twist of $\omega_C$.}
By definition,
\[
\omega_C(\mathfrak{c}) := \omega_C \otimes_{\mathcal{O}_C} \mathfrak{c}.
\]
Therefore,
\[
\mathcal{H}\!om_C(\nu_*\mathcal{O}_{\widetilde{C}},\omega_C)
\;\cong\;
\omega_C(\mathfrak{c}).
\]

\medskip
\noindent
\textit{ Conclusion.}
Combining Steps 2 and 6, we obtain a canonical isomorphism
\[
\nu_*\omega_{\widetilde{C}}
\;\cong\;
\omega_C(\mathfrak{c}),
\]
as claimed.

\medskip
\noindent
\textbf{Canonicity.}
All identifications arise functorially from Grothendieck duality and the
definition of the conductor, and hence the resulting isomorphism is canonical.
\end{proof}

\bigskip


\section{Appendix F}

\subsection*{Detailed Explanation of  Remark \ref{remark-trop1}.}

We explain the following remark \emph{line by line}, unpacking every mathematical
idea it contains.

\begin{remark}
Residue spanning guarantees that the tropical curve captures all first-order
information of differentials on the degenerating algebraic curve. In this sense,
the tropical Jacobian is not merely a combinatorial shadow, but a faithful
linearized limit of the classical Jacobian.
\end{remark}

\subsection*{Background setting}

We are working in a degeneration framework:
\begin{itemize}
  \item[(a)] A smooth algebraic curve $C_t$ degenerates, as $t \to 0$, to a singular
  (typically nodal) curve $C_0$.
  \item[(b)] Associated to this degeneration is a \emph{tropical curve} $\Gamma$,
  which is a metric graph encoding the combinatorial and valuation-theoretic
  data of the degeneration.
  \item[(c)] We study how holomorphic or meromorphic differentials on $C_t$
  behave in the limit, and how this behavior is reflected on $\Gamma$.
\end{itemize}

\medskip

\paragraph{Residues.}
When a smooth curve degenerates to a nodal curve, differentials on the smooth
fibers limit to \emph{logarithmic differentials} on the central fiber.
At each node, such a differential has two local branches with residues
$r$ and $-r$.

\paragraph{Residue data as linear information.}
Collecting all residues at the nodes produces a vector
\[
(r_e)_{e \in E(\Gamma)} \in \mathbb{R}^{E(\Gamma)},
\]
where each edge $e$ of the tropical curve corresponds to a node of the
degenerating algebraic curve.

\vspace{0.2cm}

The \emph{residue spanning condition} means that every admissible residue vector
satisfying the balancing conditions
\[
\sum_{e \ni v} \pm r_e = 0 \quad \text{at every vertex } v
\]
arises from an actual differential on the algebraic curve.
Equivalently, residues span the full space of harmonic $1$-forms on $\Gamma$.

\paragraph{Guarantees.}
Thus, ``residue spanning guarantees'' means that there is no loss of linear data
when passing from algebraic differentials to tropical ones.

\medskip

{\em First-order information} refers to the \emph{linearized} behavior of differentials:
periods, residues, and infinitesimal variations, as opposed to nonlinear or
higher-order corrections.

\paragraph{\it Differentials and periods.}
For a smooth curve $C_t$, a differential $\omega_t$ defines a linear functional
on homology via integration:
\[
\gamma \mapsto \int_\gamma \omega_t.
\]
As $t \to 0$, these periods degenerate, but their leading-order behavior is
controlled by residues and edge lengths on $\Gamma$.

\paragraph{\it Captured by the tropical curve.}
The tropical curve $\Gamma$ encodes:
\begin{enumerate}
  \item edge lengths (coming from valuations),
  \item harmonicity conditions at vertices,
  \item residue data along edges.
\end{enumerate}
Together, these encode exactly the same linear constraints satisfied by
degenerating differentials.

\vspace{0.2cm}

\paragraph{\it Degenerating curve.}
The phrase emphasizes that we are not studying a fixed singular curve, but a
\emph{family} whose smooth members carry rich analytic data.

\vspace{0.2cm}
\paragraph{\it Compatibility of limits.}
The claim is that the tropicalization process preserves the full linear structure
of the space of differentials in the limit.

\vspace{0.3cm}

\paragraph{\it Classical Jacobian.}
The Jacobian $\mathrm{Jac}(C_t)$ is a complex torus obtained from periods of
differentials:
\[
\mathrm{Jac}(C_t) = H^0(C_t,\Omega^1)^* / H_1(C_t,\mathbb{Z}).
\]

\vspace{0.2cm}
\paragraph{\it Tropical Jacobian.}
The tropical Jacobian $\mathrm{Jac}(\Gamma)$ is defined as
\[
\mathrm{Jac}(\Gamma)
= H^1(\Gamma,\mathbb{R}) / H^1(\Gamma,\mathbb{Z}),
\]
a real torus built from harmonic $1$-forms on the graph.

\vspace{0.2cm}
\paragraph{\it Not merely combinatorial.}
Calling it a ``combinatorial shadow'' would suggest it only remembers
topological or discrete data.
Residue spanning shows this is false: the tropical Jacobian remembers the full
linear limit of the period lattice.

\vspace{0.2cm}

\paragraph{\it Linearized limit.}
As $t \to 0$, the Jacobians $\mathrm{Jac}(C_t)$ degenerate.
After suitable rescaling, their limit is governed by linear data:
residues, edge lengths, and harmonic forms.

\vspace{0.2cm}

\paragraph{\it Faithfulness.}
``Faithful'' means:
\begin{enumerate}
  \item no linear relations are lost,
  \item no spurious relations are introduced,
  \item the tropical Jacobian is canonically identified with the limit of
  $\mathrm{Jac}(C_t)$ at the level of first-order geometry.
\end{enumerate}

\vspace{0.2cm}

\paragraph{\it Conclusion.}
Thus, the remark asserts that tropical geometry does not merely approximate the
classical theory, but precisely captures its linear degeneration through the
mechanism of residue spanning.


\section{Appendix G}

\subsection*{Proof that $\omega_C$ is generated by $\displaystyle \omega=\frac{dx}{\partial f/\partial y}$ on the smooth locus}

We give a complete, line-by-line proof of the standard fact that for a plane
curve, the dualizing sheaf is locally generated on the smooth locus by the
differential $\frac{dx}{\partial f/\partial y}$.


\subsection*{1. Setup}

Let $C \subset \mathbb{A}^2$ be a reduced plane curve defined by a single equation
\[
C : f(x,y)=0,
\]
where $f \in \mathbb{C}[x,y]$.
We assume that $C$ is smooth at the point under consideration.

Let $U \subset C$ be an open subset contained in the smooth locus
$C_{\mathrm{sm}}$.


\subsection*{2. The dualizing sheaf on a smooth curve}

On a smooth curve, the dualizing sheaf coincides with the canonical bundle.
Thus, on $C_{\mathrm{sm}}$,
\[
\omega_C \cong \Omega^1_{C_{\mathrm{sm}}},
\]
the sheaf of Kähler differentials on $C_{\mathrm{sm}}$.

Therefore, to describe $\omega_C$ locally on $C_{\mathrm{sm}}$, it suffices to
describe a generator of $\Omega^1_{C_{\mathrm{sm}}}$.


\subsection*{3. Kähler differentials of a plane curve}

Let $A = \mathbb{C}[x,y]/(f)$ be the coordinate ring of $C$.
The module of Kähler differentials $\Omega^1_A$ fits into the exact sequence
\[
A \cdot df
\;\longrightarrow\;
A\,dx \oplus A\,dy
\;\longrightarrow\;
\Omega^1_A
\;\longrightarrow\;
0,
\]
where
\[
df = \frac{\partial f}{\partial x}dx + \frac{\partial f}{\partial y}dy.
\]

Thus $\Omega^1_A$ is generated by $dx$ and $dy$, subject to the single relation
\[
\frac{\partial f}{\partial x}dx
+
\frac{\partial f}{\partial y}dy
=
0.
\]


\subsection*{4. Solving the relation on the smooth locus}

Let $p \in C_{\mathrm{sm}}$.
By definition of smoothness, the gradient of $f$ does not vanish at $p$, i.e.
\[
\left(
\frac{\partial f}{\partial x}(p),
\frac{\partial f}{\partial y}(p)
\right)
\neq (0,0).
\]

Without loss of generality, we may assume
\[
\frac{\partial f}{\partial y}(p) \neq 0.
\]
Hence $\frac{\partial f}{\partial y}$ is a unit in the local ring
$\mathcal{O}_{C,p}$.

The relation
\[
\frac{\partial f}{\partial x}dx
+
\frac{\partial f}{\partial y}dy
=
0
\]
can therefore be solved for $dy$:
\[
dy
=
-
\frac{\partial f/\partial x}{\partial f/\partial y}
\,dx.
\]

Thus every Kähler differential on $C_{\mathrm{sm}}$ can be written uniquely as
a multiple of $dx$.


\subsection*{5. A local generator of $\Omega^1_{C_{\mathrm{sm}}}$}

Define the differential
\[
\omega := \frac{dx}{\partial f/\partial y}.
\]

Since $\partial f/\partial y$ is a unit on $U$, the expression $\omega$ is a
regular differential on $U$.

Moreover, any differential $\alpha \in \Omega^1_{C_{\mathrm{sm}}}(U)$ can be
written as
\[
\alpha = g \, dx = g\,\frac{\partial f}{\partial y}\,\omega
\]
for some regular function $g$ on $U$.
Thus $\omega$ generates $\Omega^1_{C_{\mathrm{sm}}}$ as an
$\mathcal{O}_{C_{\mathrm{sm}}}$--module.


\subsection*{6. Independence of the coordinate choice}

If instead $\frac{\partial f}{\partial x}$ is nonzero at $p$, then one may solve
the relation for $dx$ and obtain the alternative generator
\[
\omega = -\frac{dy}{\partial f/\partial x}.
\]

These two expressions agree up to multiplication by a unit, so they define the
same invertible sheaf $\omega_C$ on the smooth locus.


\subsection*{7. Conclusion}

We have shown that:
\begin{enumerate}
  \item on the smooth locus of the plane curve $C$, the dualizing sheaf
        $\omega_C$ coincides with the canonical bundle;
  \item the canonical bundle is generated locally by Kähler differentials;
  \item whenever $\partial f/\partial y \neq 0$, a local generator is given by
        $\frac{dx}{\partial f/\partial y}$.
\end{enumerate}

Therefore, on the smooth locus of $C$, the dualizing sheaf $\omega_C$ is
generated by the regular differential
\[
{
\omega = \frac{dx}{\partial f/\partial y}.
}
\]


\subsection*{Proof that a local generator of $\omega_C$ is $\displaystyle \omega=\frac{dt}{t^2}$}

We give a complete and rigorous proof for the cuspidal singularity, written
line-by-line and suitable for inclusion in a research paper.


\subsection*{1. The cuspidal curve and its normalization}

Let $C \subset \mathbb{P}^2$ be the plane curve defined by
\[
C:\quad y^2 = x^3 .
\]
The curve $C$ is irreducible and has a unique singular point at $(0,0)$, which is
a cusp.

Let
\[
\nu \colon \widetilde{C} \longrightarrow C
\]
be the normalization of $C$.  Then $\widetilde{C} \cong \mathbb{P}^1$, and on the
affine chart $\widetilde{C}\setminus\{\infty\}$ the normalization map is given by
\[
x = t^2,
\qquad
y = t^3,
\]
where $t$ is a local coordinate on $\widetilde{C}$.
The unique point lying over the cusp is $t=0$.


\subsection*{2. The dualizing sheaf for a plane curve}

Let $C \subset \mathbb{A}^2$ be a reduced plane curve defined by a polynomial
$f(x,y)=0$.
On the smooth locus of $C$, the dualizing sheaf $\omega_C$ is generated by the
regular differential
\[
\omega = \frac{dx}{\partial f/\partial y}.
\]

In our case,
\[
f(x,y) = y^2 - x^3,
\qquad
\frac{\partial f}{\partial y} = 2y.
\]
Hence, on the smooth locus of $C$,
\[
\omega = \frac{dx}{2y}.
\]

This expression defines a regular differential on $C_{\mathrm{sm}}$ and extends
to a generator of the dualizing sheaf $\omega_C$.


\subsection*{3. Pullback to the normalization}

We now pull back $\omega$ to the normalization $\widetilde{C}$ using the map
$\nu$.
Using the parametrization $x=t^2$, $y=t^3$, we compute
\[
dx = 2t\,dt,
\qquad
y = t^3.
\]
Substituting into $\omega = \frac{dx}{2y}$ gives
\[
\nu^*\omega
=
\frac{2t\,dt}{2t^3}
=
\frac{dt}{t^2}.
\]

Thus the pullback of a generator of $\omega_C$ to the normalization is the
meromorphic differential $\frac{dt}{t^2}$.


\subsection*{4. Pole order and admissibility}

The differential $\frac{dt}{t^2}$ has a pole of order exactly $2$ at $t=0$ and no
other poles on $\widetilde{C}$.

This pole order is forced by the geometry:
\begin{itemize}
  \item[(i)] the cusp has $\delta$--invariant $\delta=1$,
  \item[(ii)] the dualizing sheaf satisfies
        \(
        \omega_C \cong \nu_*\bigl(\omega_{\mathbb{P}^1}(2\cdot 0)\bigr),
        \)
  \item[(iii)] hence sections of $\omega_C$ correspond to differentials on
        $\widetilde{C}$ with poles of order at most $2$ at $t=0$.
\end{itemize}

The differential $\dfrac{dt}{t^2}$ therefore represents a local generator of
$\omega_C$.


\subsection*{5. Conclusion}

We have shown that:
\begin{itemize}
  \item[(i)] on the smooth locus of $C$, the dualizing sheaf $\omega_C$ is generated
        by $\frac{dx}{2y}$;
  \item[(ii)] pulling back to the normalization yields $\frac{dt}{t^2}$;
  \item[(iii)] this differential has precisely the pole order dictated by the
        $\delta$--invariant of the cusp.
\end{itemize}

Therefore, on the affine chart of $\mathbb{P}^1$ with coordinate $t$, a local
generator of the dualizing sheaf $\omega_C$ is given by
\[
{
\omega = \frac{dt}{t^2}.
}
\]





\end{document}